\documentclass[11pt]{article}
\usepackage{graphicx} 
\usepackage{amsfonts,amsmath,amssymb,latexsym,psfrag}
\usepackage{amsthm}
\usepackage{dsfont}
\usepackage[colorlinks=true, urlcolor=blue,linkcolor=blue, citecolor=blue]{hyperref}
\usepackage{geometry}
\usepackage{caption}
\usepackage{booktabs}
\usepackage{subcaption}
\newtheorem{Theorem}{Theorem}[section]
\newtheorem{Proposition}[Theorem]{Proposition}

\newtheorem{Corollary}[Theorem]{Corollary}
\newtheorem{Lemma}[Theorem]{Lemma}
\newtheorem{Remark}[Theorem]{Remark}

\newtheorem{Assumption}[Theorem]{Assumption}

\usepackage{mathtools}

\newcommand{\R}{{\mathbb R}}

\newcommand{\N}{{\mathbb N}}

\newcommand{\Z}{{\mathbb Z}}
\newcommand{\F}{{\mathcal F}}

\newcommand{\E}{{\mathbb E}}

\title{The Markov approximation of the periodic multivariate Poisson autoregression}
\author{Mahmoud Khabou\footnote{Corresponding author: m.khabou@imperial.ac.uk}, Edward A. K. Cohen and Almut E. D. Veraart \\ \\
Department of Mathematics, \\ Imperial College London, \\ 180 Queen's Gate, \\ London, SW7 2AZ, \\ United Kingdom}
\date{}

\begin{document}

\maketitle
\abstract{This paper introduces a periodic multivariate Poisson autoregression with potentially infinite memory, with a special focus on the network setting. Using contraction techniques, we study the stability of such a process and provide upper bounds on how fast it reaches the periodically stationary regime. We then propose a computationally efficient Markov approximation using the properties of the exponential function and a density result. Furthermore, we prove the strong consistency of the maximum likelihood estimator for the Markov approximation and empirically test its robustness in the case of misspecification. Our model is applied to the prediction of weekly Rotavirus cases in Berlin, demonstrating superior performance compared to the existing PNAR model.}

\section{Introduction}
With the recent surge in the availability of integer-valued data, there has been a growing interest in the modelling and inference of count time series. Some early contributions to count time series include discrete ARMA (DARMA) processes, see \cite{JL1978a,JL1978b}, and  the $\{0,1\}-$valued $g-$functions introduced by Berbee \cite{Berbee}. Also, the idea of using thinning operations, see \cite{SvH1979}, for model construction is now very popular; e.g.~the  thinning-based INteger AutoRegressive models of order $p$ (INAR($p$))  were introduced by Al-Osh and Alzaid \cite{AA}, and were subsequently extended to the multivariate case by Latour \cite{Latour} and to infinite order by Kirchner \cite{Kirchner}.

More recently, there has been a growing interest in observation driven models; processes for which the dynamics is determined by their past values and a random component. One of the first examples of count observation models is the INteger Generalised AutoRegressive Conditional Heteroskedastic (INGARCH($p,q$)) model introduced in \cite{FLO}, which will serve as a building block for the models studied in this paper. A subset of these models, referred to as Poisson autoregressions, have been thoroughly studied in the literature \cite{DFT,FoTj, FRT}, where sufficient stability conditions have been established and statistical inference methods have been tested.

In \cite{mahmoud,coutin2024functionalapproximationmarkedhawkes}, Poisson autoregressions were shown to be discrete-time versions of Hawkes processes, a class of self-exciting (or inhibiting) continuous-time count processes. Since, in practice, data is often recorded on regular discrete-time intervals (\textit{e.g.} high-frequency financial data), Poisson autoregressions can be seen as a count-data-adapted version of Hawkes processes, thus retaining their usefulness in many fields such as finance \cite{EGG,BJM,ELL}, neurosciences \cite{BONNET2022109550,FPE,RRT} and social networks \cite{CS,BBGM}. Poisson autoregressions have also been applied to crime data \cite{kaur2024latentspacemodelmultivariate} and epidemiology \cite{kaur2024dynamiclatentspacetime}.

Theses applications often involve a number of interacting components (e.g. spiking neurons in the brain, shared posts on a social network) which raises the need for multivariate models. For the thinning based models, \cite{Latour}  studied a multivariate INAR($p$) process which has since been used by \cite{Kirchner_multi} for the approximate estimation of multivariate Hawkes processes using conditional least squares. For the observation driven models, Fokianos \textit{et al.} \cite{FSTD} studied both the stability and estimation of multivariate Poisson autoregressions using a Markov chain perturbation approach.  The stability condition for such autoregressions has been significantly improved by Debaly and Truquet \cite{DT}. We refer the reader to \cite{almut} for a study of multivariate (continuous-time) trawl processes and to  \cite{Fokianos_multi} for a survey on multivariate count series.

One of the challenges of general multivariate models is that they involve interaction terms between all of the components and hence the number of parameters increases rapidly with the number of particles. This means that simulation and inference can be quite infeasible for systems of a large number of components. This motivated more parsimonious approaches, which led to the incorporation of a network structure (a graph with an adjacency matrix). Building on the Network AutoRegressive (NAR) studied by Zhu \textit{et al.} \cite{Zhu} and their generalisation to $r-$stage neighbours (Knight \textit{et al.} \cite{knight2016modellingdetrendingdecorrelationnetwork}), Armillotta and Fokianos \cite{ArFok} proposed the Poisson NAR (PNAR) model for network count series. The model can be described as follows: Consider a network of $d$ nodes whose neighbourhood structure is described by an adjacency matrix $M=(m_{ij})_{i,j=1,\cdots,d}$, where $m_{ij}=1$ if there is a directed edge from node $i$ to node $j$. To each node $i$ is associated a time series of counts $(Y^{(i)}_t)_{t\in \N}$ that evolves according to the dynamics 
$$Y^{(i)}_t|\mathcal F^Y_{t-1} \sim \text{Pois}(\lambda^{(i)}_t),$$
$\mathcal F^Y$ here being the filtration associated with the network count variables. The autoregressive aspect comes from the fact that the intensity $\lambda^{(i)}_t$ is modelled as
\begin{equation}
\label{eq:PNAR}
\lambda^{(i)}_t=\mu+ \sum_{k=1}^{q}\alpha_{k}Y^{(i)}_{t-k}+ \sum_{k=1}^{q}\beta_k \frac{1}{\sum_{j=1}^d m_{ij}}\sum_{j=1}^d m_{ij}Y_{t-k}^{(j)},
\end{equation}
where regression coefficients $(\alpha_k)_{k\in1,...,q}$ are called the momentum kernel and coefficients $(\beta_k)_{k\in1,...,q}$ are called the network kernel. Assuming that these kernels are positive, the occurrence of a non-zero count $Y^{(i)}_{t-1}$ will increase the intensity $\lambda^{(i)}_t$ (as well as $\lambda^{(j)}_t$ for any node $j$ impacted by node $i$), which in turn means that $Y^{(i)}_{t}$ (and $Y^{(j)}_t$ for nodes $j$ impacted by $i$) are more likely to take larger values. The stability conditions for this process, both at the $t\to +\infty$ and $d \to +\infty$ limits, are provided in \cite{ArFok}.

We notice nevertheless that many real-life networks exhibit a seasonal (or periodic) behaviour that cannot be captured by the PNAR($q$) model because its coefficients are time invariant. For example,  the spiking behaviour of neurons changes drastically between night and day, or the number of flu cases varies across the seasons. As highlighted in \cite{BG}, systematically neglecting periodicity in financial time series leads to a loss in forecasting efficiency. We also note that classical differencing techniques used to remove seasonality before modeling are unsuitable in the count setting as the resulting differenced process will have negative values. 

We therefore propose a new model for count network autoregressions whose coefficients also vary periodically in time, that is, where the regression coefficients $(\alpha_k)_{k\in 1,\cdots, q}$ and $(\beta_k)_{k\in 1,\cdots, q}$ are replaced by the sequences $(\alpha^{(t)}_k)_{k\in 1,\cdots, q}$ and $(\beta^{(t)}_k)_{k\in 1,\cdots, q}$ that are $p-$periodic in $t$, for some integer $p$ that represents the period of the seasonality. For instance, $p=12$ months for the monthly number of storms in a given geographical area, or $p=24$ for the hourly number of posts on a social network. Because of this variation, classical concepts such as stationarity and ergodicity no longer apply and are replaced by the analogous periodic counterparts. Such concepts have been initially studied in the signal processing literature \cite{BoylesGardner,GARDNER2006639} under the name of ``cyclo-stationarity'' and ``cyclo-ergodicity'', but have since caught the attention of the time series community. To capture the seasonalities in volatility, \cite{BG} proposed a periodic ARCH model which has been used by \cite{RZ} to model the variations of gas prices. For a general book on periodic time series we refer the reader to \cite{FP}, and to \cite{Aknouche} for an article on periodic time series applied to stochastic volatility. More recently, Aknouche \textit{et al.} \cite{ABD} studied the probabilistic properties of a periodic Poisson INGARCH $(1,1)$ model, which paved the way to the study of periodic count series \cite{almohaimeed, Almohaimeed_2023,BB, SPS}.

Another property of the PNAR$(q)$ model \eqref{eq:PNAR} is its finite (and in practice short, \textit{cf.} \cite{ArFok}) memory. The fact that the intensity depends on the $q$ last counts guarantees that its correlation decays quickly and ensures that simulation from the model and computation of the likelihood function of a sample of length $T$ is $O(qT)$. However, empirical studies show that neuron spike processes \cite{locherbach}, as well as order book dynamics \cite{HBB,BJM}, have an infinite memory that sometimes decays slowly as a power law. This renders the standard PNAR$(q)$ model inadequate for capturing the long-memory properties of these processes. Of course, the immediate solution would be to work with a PNAR$(\infty)$ model, but this solution comes at a cost: the simulation and likelihood computation for a sample of length $T$ is $O(T^2)$, and non-parametric estimation is effectively impossible because there is an infinite amount of coefficients to estimate. In this paper, we use the separability property of the exponential kernel to propose a Markov approximation of the PNAR$(\infty)$ model which ensures the cost of simulation and likelihood computation is linear. 

The article is organised as follows: in Section \ref{sec:model} we introduce a general multivariate periodic Poisson autoregression, for which we distinguish two types of periodicity. Expanding the results proven in \cite{DT,DW}, we give sufficient conditions for the stability of multivariate periodic autoregressions in Section \ref{section:stability}. These are then applied to the periodic multivariate Poisson autoregression to yield upper bounds on the speed of convergence of the aforementioned process to its periodically stationary regime. In Section \ref{sec:approx} we prove the continuity of the multivariate Poisson autoregression with respect to its kernel, which we then use alongside a density lemma to give a Markov approximation. Section \ref{sec:inference} deals with inference for infinite memory autoregressions, in which we first show the strong consistency of the maximum likelihood estimator (MLE) for models with exponential polynomial kernels, and then empirically examine its performance for misspecified processes. In Section \ref{sec:ral_data} we apply our model to the forecast of the weekly number of cases of Rotavirus in Berlin between 2001 and 2015, comparing it to the existing PNAR model introduced in \cite{ArFok}.
\section{The model}
\label{sec:model}
\subsection{Notations}
Throughout this article, $\Z$, $\N$ and $\N^*$ denote the sets integers, non-negative integers and positive integers, respectively. Superscripts are denoted between brackets to be distinguished from powers.

Let $\left({N}_t(\cdot) \right)_{t \in \N^*}=\left({N}^{(1)}_t(\cdot),\cdots,N^{(d)}_t(\cdot) \right)_{t \in \N^*}$ be a family of independent and identically distributed \textit{(iid)} unit intensity Poisson processes of dimension $d$, with $\mathcal F=\left(\mathcal F_t \right)_{t\in \N^*}$ their associated filtration. More specifically, for a given $X=(X^{(1)},\cdots,X^{(d)})\in \R_+^d$, $N_t(X)$ counts the number of points in $[0,X^{(1)}]\times \cdots \times [0,X^{(d)}]$. Unlike \cite{FSTD} who use copula Poisson processes, we assume that the different components of $N_t$ are independent. For two vectors $X=(X^{(1)},\cdots,X^{(d)})$ and $Y=(Y^{(1)},\cdots,Y^{(d)})$ in $\R^+$, $|X|$ denotes the vector $(|X^{(1)}|,\cdots,|X^{(d)}|)$ and $\preceq $ denotes the partial order $X\preceq Y \iff X^{(i)} \leq Y^{(i)}$ for all $i=1,\cdots,d$. The relation $\preceq$ is naturally extended to matrices.
The vector of $\R ^d$ that contains $1$ in every component is denoted by $\mathds 1$. 
Let $\psi:\R \mapsto \R_+$ be an $L-$Lipschitz function, assumed to be increasing. For $X\in \R^d, \psi(X):=\left(\psi(X^{(1)}),\cdots,\psi(X^{(d)})\right)$. We refer to $\psi$ as the \textit{jump-rate} function.\\
In what follows, $p\in \N^*$ is a fixed period. When we say that a sequence $(u_t)_{t\in \Z}$ is periodic, we mean $u_{t+p}=u_t$ for all $t\in \Z$. 

Our model describes a $d-$dimensional count time series $\left(Y_t\right)_{t\in \N}$, that is a time series that takes values in $\N^d$. We assume that $Y$ is measurable with respect to $\mathcal F$ and we denote its filtration by $\mathcal F ^Y$. For the count series $Y$, the intensity $\left(\lambda_t\right)_{t \in \N}$ denotes $\lambda_t = \E \left [Y_t\big | \mathcal F^Y_{t-1} \right]$.

We will now present autoregressive count series with two types of periodicity. 
\subsection{Type I periodic Poisson autoregression}
Given an initial  vector of counts $Y_0\in \N^d$, the multivariate Poisson autoregression is constructed recursively:
\begin{equation}
\label{def:TypeI}
    \begin{cases}
        Y_t&=N_t\left(\lambda_t\right),\\
        \lambda_t &= \psi \left(\mu_t+\sum_{k=1}^{t-1}\phi^{(t)}_{t-k}Y_k\right)\\
        &=\psi \left(\mu_t+\sum_{k=1}^{t-1}\phi^{(t)}_{k}Y_{t-k}\right)
    \end{cases}
        \text {for } t \in \N^*,
\end{equation}
where $(\mu_t)_{t \in \Z}$ is periodic family of baseline pre-intensities and $(\phi^{(t)}_k)_{t \in \Z, k \in \N^*}$ is a periodic (in $t$) family of $d\times d$ matrices that encodes the impact of the counts of lag $k$ at time $t$, called the \textit{kernel}. If the terms of $\phi^{(t)}_k$ are positive, a non-zero number of counts $Y_{t-k}$ increases the sum in \eqref{def:TypeI}, resulting in an increase in $\lambda_{t+1}$, \textit{ceteris paribus}, thus giving a higher likelihood of observing a non-zero vector of counts $Y_{t+1}$. This means that the network is overall self-exciting. Similarly, if the kernel matrices have negative entries, the network becomes self-inhibiting overall. Writing \eqref{def:TypeI} component-wise yields 
\begin{equation}
    \begin{cases}
        Y^{(i)}_t&=N^{(i)}_t\left(\lambda^{(i)}_t\right),\\
        \lambda^{(i)}_t &= \psi \left(\mu^{(i)}_t+\sum_{k=1}^{t-1}\sum_{j=1}^d\left(\phi^{(t)}_{t-k}\right)_{ij}Y^{(j)}_k\right)\\
        &=\psi \left(\mu^{(i)}_t+\sum_{k=1}^{t-1}\sum_{j=1}^d(\phi^{(t)}_{k})_{ij}Y^{(j)}_{t-k}\right)
    \end{cases}
        \text {for } t \in \N^*,
\end{equation}
showing $(\phi^{(t)}_{k})_{ij}$ quantifies the impact of node $j$ on node $i$ at time $t$ at lag $k$. If $(\phi^{(t)}_{k})_{ij}=0$, then node $j$ has no \textit{direct} impact on node $i$, although it can impact it indirectly via other nodes. 

The model can be seen as a non-linear, infinite memory and multivariate generalisation of the periodic INGARCH$(p,0)$ process recently introduced by \cite{Almohaimeed_2023}. We refer the reader to \cite{aknouche:tel-04553687} for a survey on periodic ARCH time series (in French). 

Throughout this article, we focus on the network setting, with a deterministic and constant neighbourhood structure. Based on continuous-valued network autoregressive networks models \cite{JSSv096i05, Zhu}, Armillotta and Fokianos introduced a network count autoregressive model \cite{ArFok} called the PNAR model. We deal here with a periodic and infinite memory extension of the aforementioned autoregression.

We consider a network with a fixed adjacency matrix $M=(m_{ij})_{i,j=1,\cdots,d}$, that is $m_{ij}=1$ if there is a directed edge from node $i$ to node $j$, and $m_{ij}=0$ otherwise. We impose that a node is not connected to itself, that is $m_{ii}=0$ for all $i=1,\cdots,d$. For a node $i$, the out-degree is defined as the total number of nodes that influence $i$, i.e. $n_i=\sum_{j=1}^d m_{ji}$. The Type I periodic network Poisson autoregression is given by 
\begin{equation*}
    \begin{cases}
        Y^{(i)}_t&=N^{(i)}_t(\lambda_t^{(i)})\\
        \lambda^{(i)}_t&=\psi \left(\mu^{(i)}_t+\sum_{k=1}^{t-1}\alpha^{(t)}_{t-k}Y^{(i)}_k+\beta^{(t)}_{t-k}\frac{1}{n_i}\sum_{j=1}^d m_{ij}Y^{(j)}_k\right)
    \end{cases},
\end{equation*}
where $(\alpha^{(t)}_k)$ and $(\beta^{(t)}_k)$, the momentum and network kernels, respectively, are scalar sequences that are periodic in $t$.

To cast this autoregression in vector form, we define $W=\text{diag}(n_1^{-1},\cdots, n_d^{-1})M$ to be the normalised adjacency matrix and take $$\phi^{(t)}_{t-k}=\alpha^{(t)}_{t-k}I_d + \beta^{(t)}_{t-k}W$$
in \eqref{def:TypeI}. The model can be further extended to allow for the interaction with indirect neighbours (neighbours of neighbours and so on) following the work of \cite{JSSv096i05}. We refer the reader to \cite{liu2023newmethodsnetworkcount} for a recent survey on the different methods for network count time series.

Note, the periodicity introduced in \eqref{def:TypeI} is in the current time $t$ and not the lag $k$. This is not the case for the Type II periodicity, which we now introduce.
\subsection{Type II periodic Poisson autoregression}
We now consider a model for which the season of the lagged counts is what matters in determining the effects on the autoregression. 
Given a first vector of counts $Y_0\in \N^d$, the multivariate Poisson autoregression is constructed recursively:
\begin{equation}
\label{def:TypeII}
    \begin{cases}
        Y_t&=N_t\left(\lambda_t\right),\\
        \lambda_t &= \psi \left(\mu_t+\sum_{k=1}^{t-1}\phi^{(k)}_{t-k}Y_k\right)\\
        &=\psi \left(\mu_t+\sum_{k=1}^{t-1}\phi^{(t-k)}_{k}Y_{t-k}\right)
    \end{cases}
        \text {for } t \in \N.
\end{equation}
As is the case for Type I seasonality, $(\mu_t)_{t \in \Z}$ is a periodic family of baseline pre-intensities and $(\phi^{(t)}_k)_{t \in \Z, k \in \N^*}$ is a periodic (in $t$) family of $d\times d$ matrices. Unlike Type I periodicity, Type II periodicity has not been thoroughly studied in the literature. However, the two types are still different. To illustrate the difference, we consider a 12-node network generated according to the Stochastic Block Model (SBM) with two communities, the first containing 4 nodes and the second containing 8 nodes. The probabilities $p_{ij}$ of a node from community $i$ forming an edge with a node from community $j$ are $p_{11} = 0.8$, $p_{22}=0.7$ and $p_{12} = p_{21} = 0.1$. The kernels are
\begin{equation}
\label{eq:example_kernel}
    \begin{cases}
        \alpha^{(t)}_k&=0.3 \frac{8\cdot \mathds 1_{t=0\text{ mod }[4]}+0.2 \cdot\mathds 1_{t\neq 0\text{ mod }[4]}}{k^2}\mathds 1_{k\leq 10},\\
        \beta^{(t)}_k&=0.2 e^{-3k}\left(1+\cos(t \pi/2)\right),
    \end{cases}
\end{equation}
and baseline pre-intensity is $\mu= 0.4$.

 A simulation of the Poisson autoregression with periodicity of Type I for the third node is given in Figure \ref{fig:type_I}.
\begin{figure}[h!]
    \centering
    \includegraphics[width=1\linewidth]{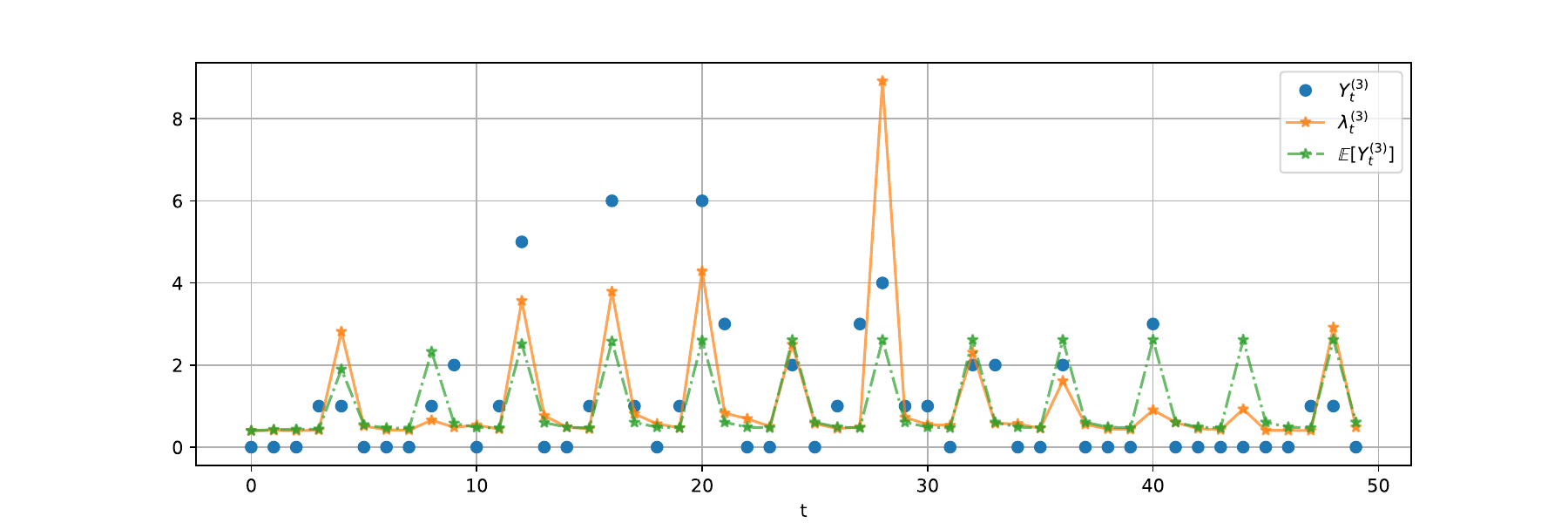}
    \caption{The counts drop significantly when $t$ is not divisible by 4.}
    \label{fig:type_I}
\end{figure}
Type II periodicity, on the other hand, means that the impact depends on the past seasons rather than the current season, as illustrated in Figure \ref{fig:type_II}. The two types of periodicity differ in the following way:
\begin{enumerate}
    \item For Type I periodicity, the network remembers its past and regresses on it now according to weights that depend on the current season. 
    \item For Type II periodicity, the network remembers its past activity weighted by coefficients that depend the past seasons, then aggregates it.
\end{enumerate}
 To the best of our knowledge, the difference between the two types of periodicity has not been explicitly studied in the literature. However, we point out that Maillard and Wintenberger \cite{MW} mentioned that their model of autoregression with random coefficients can either be applied to $\phi^{(t)}_{t-k}$ (analogous to Type I periodicity) or $\phi^{(k)}_{t-k}$ (analogous to Type II periodicity).
\begin{figure}[h!]
    \centering
    \includegraphics[width=1\linewidth]{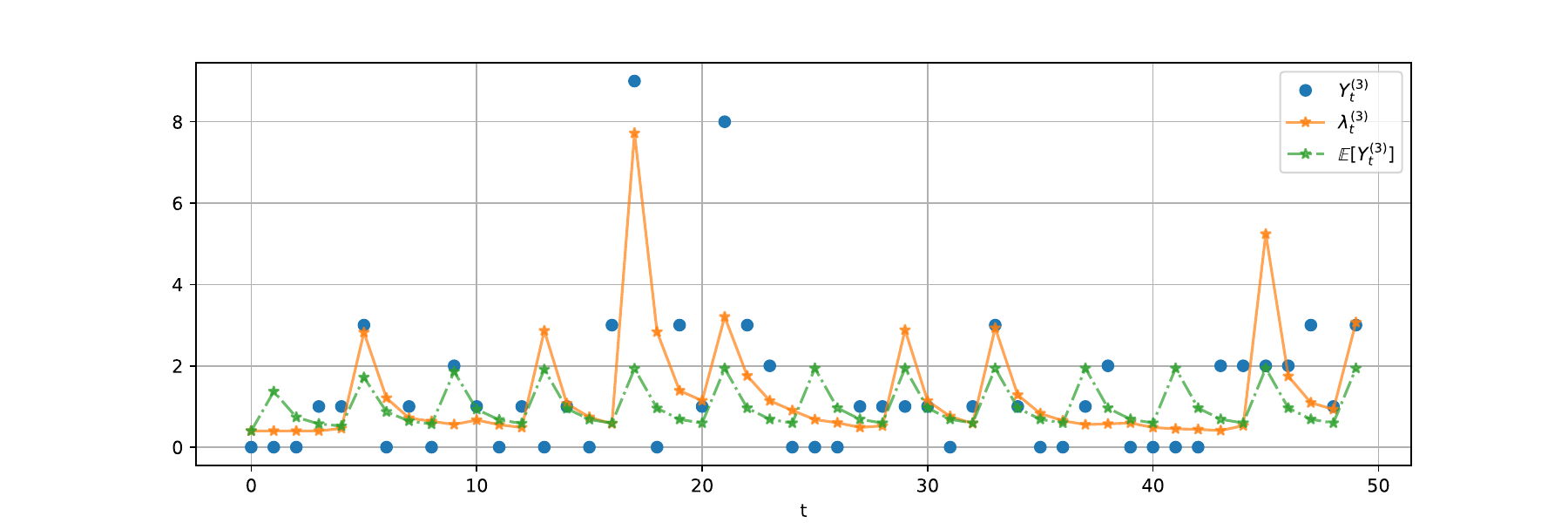}
    \caption{The peak of the average number of counts no longer coincides with the instants $t=0\text{ mod}[4]$, but occurs right after. We can also see that it decreases more slowly from its peak than in the Type I periodicity.}
    \label{fig:type_II}
\end{figure}

\begin{Remark}
    We point out that we treat the intensity $\lambda_t$ as the projection of the count variable $Y_t$ on the $\sigma-$algebra $\mathcal F^Y_{t-1}$. However, since the knowledge of the underlying randomness $\mathcal F_{t-1}=\sigma \left( N_1,\cdots,N_{t-1}\right)$ yields the information on the counts $\mathcal F^Y_{t-1}$, we have that $$\lambda_t=\E \left[Y_t|\mathcal F_{t-1}\right].$$
    Throughout the paper, we prefer conditioning on $\mathcal F^Y_{t-1}$ as it contains the information of the last observed counts. However, when we deal with couplings from the same underlying randomness (\textit{cf.} the next section), conditioning on $\mathcal F_{t-1}$ is preferred.  
\end{Remark}
\section{Periodic stability of the model}
\label{section:stability}
We seek to prove that both types of multivariate periodic Poisson autoregression converge to a periodically stationary and periodically ergodic solution if they are started from an arbitrary point. Proving this result by the direct application of the contraction argument presented in \cite{DW} is not optimal for our model, and this is for two reasons:
\begin{enumerate}
    \item The Lipschitz condition (3.1) in the aforementioned reference is given with respect to a reference Orlicz norm on the Banach space, which can be too strong for autoregressions on $\R^d$. Indeed, a sequence of matrices $(A_k)_{k\in \N}$ can satisfy $\rho(\sum_{k\geq 1}A_k)<1$ while having $\sum_{k \geq 1} \|A_k\|>1$ for the usual operator norms.
    \item Since our iterations are periodic, the contraction should be verified on a period, which means that the Lipschitz condition should be obtained for $p$ iterations of the autoregression function.
\end{enumerate}
This is why we resort to proving the different results for the periodic finite memory multivariate autoregressions, which we then extend to the infinite memory setting.
\subsection{General results for periodic infinite memory autoregressions}
\label{sec:general_results}
We study the periodic multivariate infinite autoregression 
\begin{equation}
    \label{eq:inf_reg}
    X_t=f_t \left(X_{t-1},X_{t-2},\cdots; \zeta_t \right),
\end{equation}
where $(f_t)_{t \in \Z}$ is a periodic sequence of functions from $(\R^d)^{\N}$ to $\R^d$ and $(\zeta_t)_{t\in \Z}$ is an independent and periodically distributed (\textit{i.p.d}) sequence of random variables defined on some measurable space $E$, that is $\zeta_{t+np}\stackrel{d}{=}\zeta_{t}$ for any $t \in \Z$ and $n\in \N$.

We seek to generalise the results proven in \cite{DT} in three directions: (i) the autoregression functions considered here are periodic; (ii) the autoregression can have infinite memory; (iii) the solution started at a random point is shown to converge towards the stationary regime with a given speed. To do so, we require the following stability assumptions.
\begin{Assumption}
\label{ass:contraction}
    There exists a family of nonnegative matrices $(A_k)_{k \in \N^*}$ satisfying the inequality $\rho\left( \sum_{k=1}^{+\infty}A_k\right)<1$ such that for all $v=1,\cdots,p$ we have
    $$\E \left [|f_v(x_1,x_2,\cdots;\zeta_v)-f_v(x'_1,x'_2,\cdots;\zeta_v)|_{ } \right] \preceq \sum_{k=1}^{+\infty} A_k |x_k-x'_k|_{ },$$
    for any $(x_1,x_2,\cdots)$ and $(x'_1,x'_2,\cdots)$ in $(\R^d)^{\N}$. Furthermore, assume that 
    $$\E [|f_v(0,0\cdots;\zeta_v)|_{ }] <+\infty.$$
\end{Assumption}

We follow the proofs in \cite{DT}, generalising them for the case of a periodic process with infinite memory. We start by proving periodic stationarity for the finite memory approximation. 
The $m-$truncated regression is defined by the equation 
\begin{equation}
    \label{eq:finite_reg}
    X^{(m)}_t=f_t \left(X^{(m)}_{t-1},X^{(m)}_{t-2},\cdots,X^{(m)}_{t-mp},0,\cdots; \zeta_t \right).
\end{equation}
We now introduce a periodic multivariate contraction condition that depends on the order $m$.
\begin{Assumption}
    \label{ass:contraction_periodic}
    There exists $p$ families of non-negative matrices $(A^{(v)}_k)_{k\in \N^*}$ such that for all $v=1,\cdots,p$ we have
    $$\E \left [|f_v(x_1,x_2,\cdots,x_{mp},0,\cdots;\zeta_v)-f_v(x'_1,x'_2,\cdots,x'_{mp},0,\cdots;\zeta_v)|_{ } \right] \preceq \sum_{k=1}^{mp} A^{(v)}_k |x_k-x'_k|_{ }.$$
    For $v=1,\cdots,p$, let $$\Gamma_v = \begin{pmatrix}
        A^{(v)}_1 & A^{(v)}_2 & \cdots &A^{(v)}_{mp-1} &A^{(v)}_{mp}\\
    1 & 0 & \cdots & 0 &0 \\
    \vdots & \ddots & \ddots & \vdots & \vdots \\
    0 & 0 & \cdots & 1 &0 
    \end{pmatrix}$$
    be the companion matrix of $(A^{(v)}_k)_{k=1,\cdots,mp}$. Assume that $\rho \left (\Gamma_p \Gamma_{p-1}\cdots \Gamma_1 \right)<1.$ Furthermore, assume that 
    $$\E [|f_v(0,0\cdots;\zeta_v)|_{ }] <+\infty.$$
\end{Assumption}
Note that Assumption \ref{ass:contraction} is stronger than Assumption \ref{ass:contraction_periodic}, in the sense that the first condition imposes that the function $f_t$ is contractive along every season, whereas the second only imposes that we have a contraction over a period; \textit{cf.} \cite{ABD} for a discussion about this condition in the context of a simpler autoregression. 

More rigorously, if Assumption \ref{ass:contraction} holds, then we have that for any given $m \in \N^*$, $\rho\left(\sum_{k=1}^{mp}A_k\right)<1 $. Lemma 1 in \cite{DT} guarantees that $\rho(\Gamma_p\Gamma_{p-1}\cdots \Gamma_1) = \rho(\Gamma^p)<1$ where 
$$\Gamma= \Gamma_v = \begin{pmatrix}
        A_1 & A_2 & \cdots &A_{mp-1} &A_{mp}\\
    1 & 0 & \cdots & 0 &0 \\
    \vdots & \ddots & \ddots & \vdots & \vdots \\
    0 & 0 & \cdots & 1 &0 
    \end{pmatrix}.$$

Before proving the stability results for the finite memory approximation, we recall the definitions of periodic stationarity and periodic weak dependence. We say that the process $(X_t)_{t\in \Z}$ is periodically stationary (resp.~periodically weakly dependent) if the seasonally embedded vector \cite{Glad,TG} $$\left(Z_n=\left(X_{np+p},X_{np+p-1},\cdots,X_{np+1}\right)\right)_{n\in \Z}$$
is stationary (resp.~weakly dependent \cite{DP}) in the usual sense. In particular, periodic stationarity means that the distribution of 
$(X_t)_{np+p}$ is invariant under any shift that is a multiple of the period $p$. We refer the reader to \cite{Aknouche} for an in-depth discussion of those concepts.  
\begin{Proposition}
\label{prop:truncate}
    Let $m\in\N^*$ and $(\zeta)_{t\in \Z}$ be an i.p.d family of random variables. Under Assumption \ref{ass:contraction_periodic}, the regression \eqref{eq:finite_reg} has a unique periodically stationary and periodically weakly dependent solution $(\tilde X^{(m)}_t)_{t\in \Z}$. Moreover, if for a given fixed history $(x_{0},x_{-1},\cdots)$, we set $ X^{(m)}_t=x_{t}$, for $t\leq 0$, and 
    $$ X^{(m)}_t=f_t \left( X^{(m)}_{t-1}, X^{(m)}_{t-2},\cdots, X^{(m)}_{t-mp},0,\cdots; \zeta_t \right), \quad \text{for }t>0,$$
    then there exists $C>0$ and $r\in (0,1)$ such that
    $$\E \left [|\tilde X^{(m)} _t - X^{(m)}_t|_{ }\right] \preceq C r^{t}, \quad \text {for $t \geq 0$}.$$
\end{Proposition}
\begin{proof}
  The proof can be found in Appendix \ref{sec:prop:truncate}.
\end{proof}
The generalisation of the existence of a unique periodically stationary solution to \eqref{eq:finite_reg} is established in the following theorem. We also prove, in the proposition that follows, along the lines of \cite{DW}, that the process started from an arbitrary history will converge towards the periodically stationary solution.
\begin{Theorem}
\label{thm:infinite}
    Let $(\zeta_t)_{t\in \Z}$ be an i.p.d. family of random variables and $(f_t)_{t\in \Z}$ be a periodic sequence of functions from $(\mathbb R ^d)^{\N}$ to $\R^d$ satisfying Assumption \ref{ass:contraction}. There exists a unique periodically stationary and periodically weakly dependent time series $(\tilde X_t)_{t\in \Z}$ that solves 
    $$\tilde X_t=f_t(\tilde X_{t-1},\tilde X_{t-2},\dots; \zeta_t).$$
\end{Theorem}
\begin{proof}
The proof can be found in Appendix \ref{sec:thm:infinite}
\end{proof}
We now give an upper bound on the speed of decay of the distance between a solution started with a given history $X$ and the stationary regime $\tilde X$. Throughout this section, $*$ denotes the convolution product $(a * b)_t=\sum_{k=1}^{t-1}a_kb_{t-k}$ for the nonnegative sequences of matrices $(a_k)$ and $(b_k)$ defined on $\N^*$. By associativity, we can define recursively $a^{*1}=a$ and $a^{*(n+1)}=a*a^{*n}$. Given the matrix sequence $(A_k)_{k\in \mathbb N^*}$ from Assumption \ref{ass:contraction}, we define $B=\sum_{n\geq 1}A^{*n}$, which is in $\ell_1(\N^*)$. We also define the matrix remainder sequence $U_t=\sum_{k=t}^{+\infty} A_k$, which clearly tends to zero as $t\to +\infty$. Set $(x_0,x_{-1},x_{-2},\cdots)$ to be a bounded sequence in $\R^d$ and define $(X_t)_{t\in \N}$ recursively by
\begin{equation}
\label{eq:given_history}
 X_t=f_t \left( X_{t-1},\cdots, X_1,x_0,x_{-1},\cdots; \zeta_t\right).
\end{equation}
\begin{Proposition}
\label{prop:speed}
    Assume that \ref{ass:contraction} holds and let $\tilde X$ be the unique periodically stationary solution of \eqref{eq:inf_reg}. Let $ X$ be the solution of the regression \eqref{eq:given_history} with a given bounded history $(x_0,x_{-1},\cdots)$. We then have that for any $t \in \N$
    $$\E \left [|X_t-\tilde X_t|_{ } \right]\preceq \left(\sum_{k=1}^{t} B_kU_{t-k}\right)C \xrightarrow[t\to +\infty] {}0,$$
    where $C$ is a nonnegative constant vector. More specifically,
    \begin{itemize}
        \item If for some $\beta >0$ we have $A_k = O(e^{-\beta k})$, then there exists a $\delta \in (0,\beta)$ such that  $$\E \left [|\tilde X_ t -X_t|_{ } \right] \preceq  C e^{-\delta t}, \quad \text{for all $t \in \N^*$}.$$
        \item If for some $\beta >0$ we have $A_k = O(k^{-2(1+\beta)})$, then
        $$\E \left [|\tilde X_ t -X_t|_{ } \right] \preceq  \frac{C}{t},  \quad \text{for all $t \in \N^*$}.$$
    \end{itemize}
\end{Proposition}
\begin{proof}
  The proof can be found in Appendix \ref{sec:prop:speed}
\end{proof}
While Proposition \ref{prop:truncate} can be seen as a matrix-adapted generalisation of Proposition 3.1 in \cite{DW} to autoregressions with periodic coefficients, it also has the merit of providing more explicit upper bounds on the speed of convergence of the process started from a given initial history to its periodically stationary limit. 

This section concludes with a result on the almost sure vanishing of the difference between the periodically stationary trajectory and the trajectory started with an arbitrary history, provided that the kernel vanishes exponentially fast. 
\begin{Corollary}
\label{cor:almost_sure}
 Assume that \ref{ass:contraction} holds and let $\tilde X$ be the unique periodically stationary solution of \eqref{eq:inf_reg}. Let $X$ be the solution of the regression \eqref{eq:given_history} with a given bounded history $(x_0,x_{-1},\cdots)$. Furthermore, assume that the for some $\beta >0$ we have that $A_k=O(e^{-\beta k})$. Then, almost surely, there exists a constant $C>0$ and a $\delta>0$ such that for all $t \in \N$
 $$|\tilde X_t-X_t| \preceq C e^{-\delta t}.$$
\end{Corollary}
\begin{proof}
The proof can be found in Appendix \ref{sec:almost_sure}.\end{proof}

\subsection{Application to multivariate periodic Poisson autoregressions}
We now apply the results of Section \ref{sec:general_results} to give sufficient conditions for the existence of periodically stationary, ergodic and weakly dependent multivariate Poisson autoregressions. We prove that, if the kernel matrices of models \eqref{def:TypeI} and \eqref{def:TypeII} are bounded by sequences of matrices satisfying Assumption \ref{ass:contraction} or \ref{ass:contraction_periodic}, then the Poisson autoregression converges to its periodically stationary, ergodic and weakly dependent version. These periodically stationary, ergodic and weak dependent autoregressions are solutions of the autoregressions 
\begin{equation}
\label{def:TypeI_s}
    \begin{cases}
        \tilde Y_t&=N_t\left(\tilde \lambda_t\right),\\
        \tilde \lambda_t &= \psi \left(\mu_t+\sum_{k=-\infty}^{t-1}\phi^{(t)}_{t-k}\tilde Y_k\right)\\
        &=\psi \left(\mu_t+\sum_{k=1}^{+\infty}\phi^{(t)}_{k}\tilde Y_{t-k}\right)
    \end{cases}
        \text {for } t \in \Z,
\end{equation}
for Type I periodicity, and 
\begin{equation}
\label{def:TypeII_s}
    \begin{cases}
        \tilde Y_t&=N_t\left(\tilde\lambda_t\right),\\
        \tilde \lambda_t &= \psi \left(\mu_t+\sum_{k=-\infty}^{t-1}\phi^{(k)}_{t-k}\tilde Y_k\right)\\
        &=\psi \left(\mu_t+\sum_{k=1}^{+\infty}\phi^{(t-k)}_{k}\tilde Y_{t-k}\right)
    \end{cases}
        \text {for } t \in \Z,
\end{equation}
for Type II periodicity. Note that we here use the same Poisson processes $N$ used for the construction of \eqref{def:TypeI} and \eqref{def:TypeII}.

We now state a sufficient stability condition for the periodic multivariate Poisson autoregression.
\begin{Assumption}
\label{ass:stability}
    There exists a family of matrices $(A_k)_{k\in \N^*}$ with non-negative coefficients satisfying 
    $$L|\phi^{(v)}_k| \preceq A_k \text{ for any }v=1,\cdots,p,\quad  k\in \N^*  \text{ and} \quad \rho\left(\sum_{k=1}^{+\infty} A_k \right)<1.$$
\end{Assumption}
This assumption is satisfied for instance if the kernel at each season $\phi^{(v)}$ is an attenuation of a global kernel $A$, \textit{e.g.} 
$$\phi^{(v)}_k=\frac{\sin(2\pi v/p)}{L}A_k.$$

\begin{Proposition}
\label{prop:stationary}
Let $(\phi_k^{(t)})_{t\in \Z, k \in \N^*}$ be a family of kernel matrices periodic in $t$ and let $\psi$ be an $L-$Lipschitz non-negative jump-rate function. 
    Assume that $(Y_t)_{t\in \N}$ is a multivariate Poisson autoregression that follows the recursion \eqref{def:TypeI} or \eqref{def:TypeII}. If Assumption \ref{ass:stability} holds, then
    \begin{enumerate}
        \item Both equations \eqref{def:TypeI_s} and \eqref{def:TypeII_s} admit a unique solution $(\tilde Y)_{t\in \Z}$ that is strictly periodically stationary, ergodic and weakly dependent.
        \item For $Y$ constructed according to recursion \eqref{def:TypeI} or \eqref{def:TypeII} (that is, with an empty history), we have that $$\E [|\tilde Y_t-Y_t|]\xrightarrow[t\to +\infty] {}0.$$
        More specifically, if $(A_k)_{k\in \N^*}$ decays exponentially (respectively at a speed $k^{-2(1+\beta)}$ for some $\beta >0$) then the convergence happens at least at an exponential rate (respectively at least at the rate $k^{-1}$). 
        \item For any $r\geq 1$ and $t\in \N^*$, we have that $\E[Y^r_t]\preceq C$ where $C$ is a non-negative constant that does not depend on $t$. 
    \end{enumerate}
\end{Proposition}
\begin{proof}
The proof can be found in Appendix \ref{sec:prop:stationary}.
\end{proof}
The same proof can be used to deduce the convergence towards a periodic stationary and weakly dependent solution for the multivariate periodic Poisson autoregression with finite memory under the less restricting stability condition \ref{ass:contraction_periodic}. 
\begin{Proposition}
    Given $m \in \N^*$, let $(\phi^{(t)}_k)_{t\in \Z, k\in \{1,\cdots,mp\}}$ be a family of kernel matrices periodic in $t$ and let $\psi$ be an $L-$Lipschitz non-negative jump-rate function. Assume that $(Y_t)_{t\in \N}$ is a multivariate Poisson autoregression of Type I following recursion \eqref{def:TypeI}. For $v\in\{1,\cdots,p\}$, let 
    $$\Gamma_v = \begin{pmatrix}
        |\phi^{(v)}_1| & |\phi^{(v)}_2| & \cdots &|\phi^{(v)}_{mp-1}| &|\phi^{(v)}_{mp}|\\
    1 & 0 & \cdots & 0 &0 \\
    \vdots & \ddots & \ddots & \vdots & \vdots \\
    0 & 0 & \cdots & 1 &0 
    \end{pmatrix}.$$
    If $\rho(\Gamma_1 \cdots \Gamma_p)<1$, then 
        \begin{enumerate}
        \item Equation \eqref{def:TypeI_s} admits a unique solution $(\tilde Y)_{t\in \Z}$ that is strictly periodically stationary, ergodic and weakly dependent.
        \item For any $r\geq 1$ and $t\in \N^*$, we have that $\E[Y^r_t]\preceq C$, where $C$ is a non-negative constant that does not depend on $t$. 
    \end{enumerate}
\end{Proposition}
Assumption \ref{ass:stability} is not optimal in two ways:
\begin{enumerate}
    \item It implies contraction along all seasons and not the weaker contraction over a period. For example, for the periodic INGARCH($1,1$) model proposed in \cite{ABD} --- which is equivalent to a Type II periodic Poisson autoregression with an exponential kernel $\phi^{(t)}_k=\nu_{t+1} e^{-\frac{k}{\tau}}$ as we shall see in Remark \ref{remark:abd} --- the sufficient stability condition is 
    $$Le^{-\frac{1}{\tau}}\prod_{v=1}^p (1+\nu_v)<1.$$
    Whereas, Assumption \ref{ass:stability} would necessitate the stronger condition
    $$Le^{-\frac{1}{\tau}}(1+\nu_v)<1, \quad \text{for all }v=1,\cdots,p.$$
    \item Just like their continuous-time counterparts, Hawkes processes, the instability of Poisson autoregressions stems from auto-/cross-excitation. Thus, assuming that $\psi$ is non decreasing, the negative values of $\phi^{(v)}$ should not have an impact on the stability, and we expect that Assumption \ref{ass:stability} could be relaxed to a condition on the positive part $(\phi^{(v)})_+$ rather than the absolute value $|\phi^{(v)}|$. This has been proven for Hawkes processes with $\psi(x)=(x)_+$ in \cite{CGMT}, using non-trivial renewal techniques. 
\end{enumerate}
    To illustrate the convergence results of Proposition \ref{prop:stationary}, we simulate a linear Type II $4-$periodic Poisson autoregression (network of size $d=12$) with kernels of the same form as \eqref{eq:example_kernel}.
    We compare the version with empty history (that is, following 
    \eqref{def:TypeI}) to the stationary solution (that is, following \eqref{def:TypeI_s}). The trajectories are shown on Figure \ref{fig:stationary}.

    \begin{figure}[h!]
        \centering
        \includegraphics[width=1\linewidth]{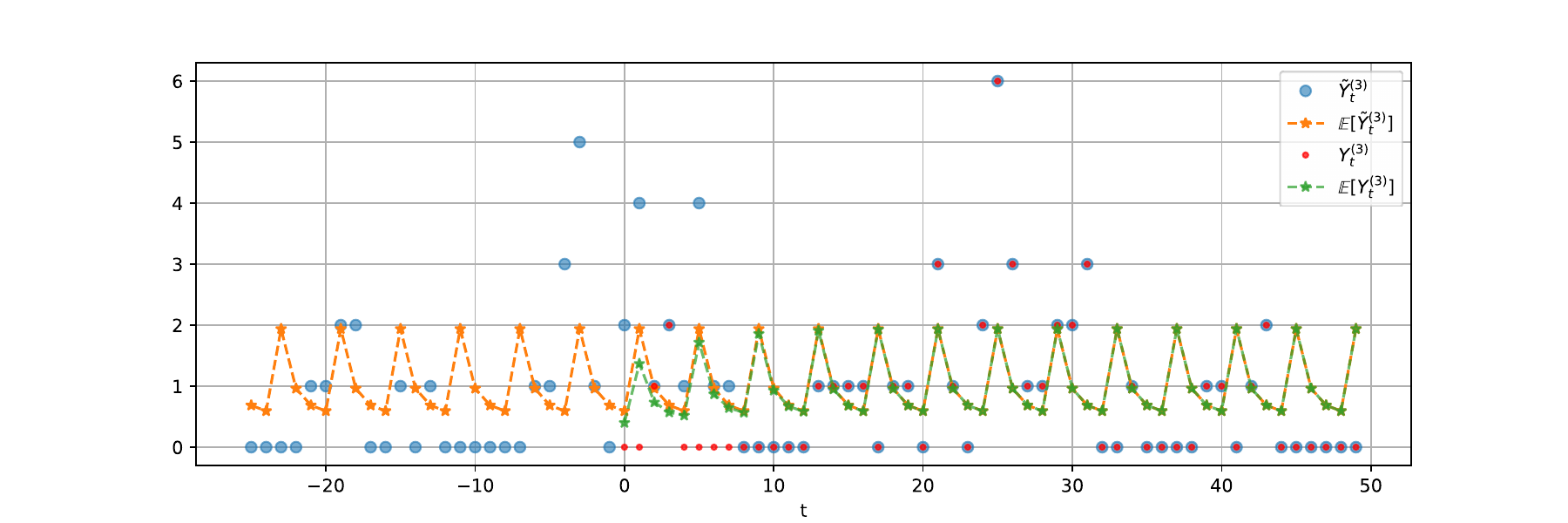}
        \caption{The counts $Y$ differ occasionally from the stationary solution $\tilde Y$ for $t\in \{0,\cdots,7\}$ but then the two become identical. Similarly, after a transitory period the curve $\E [Y^{(3)}_t]$ joins the strictly periodic curve $\E[\tilde Y^{(3)}_t]$.}
        \label{fig:stationary}
    \end{figure}

    For our example, the kernel matrices are dominated by $Ce^{-3k}$ for some positive matrix $C$, hence we expect that the difference between the empty history time series $Y$ and its periodically stationary version $\tilde Y$ decay at least exponentially fast. To illustrate this, we plot the logarithm of $\E |\tilde Y_t - Y_t|$ as a function of time on Figure \ref{fig:decay}.

    \begin{figure}[h!]
        \centering
        \includegraphics[width=1\linewidth]{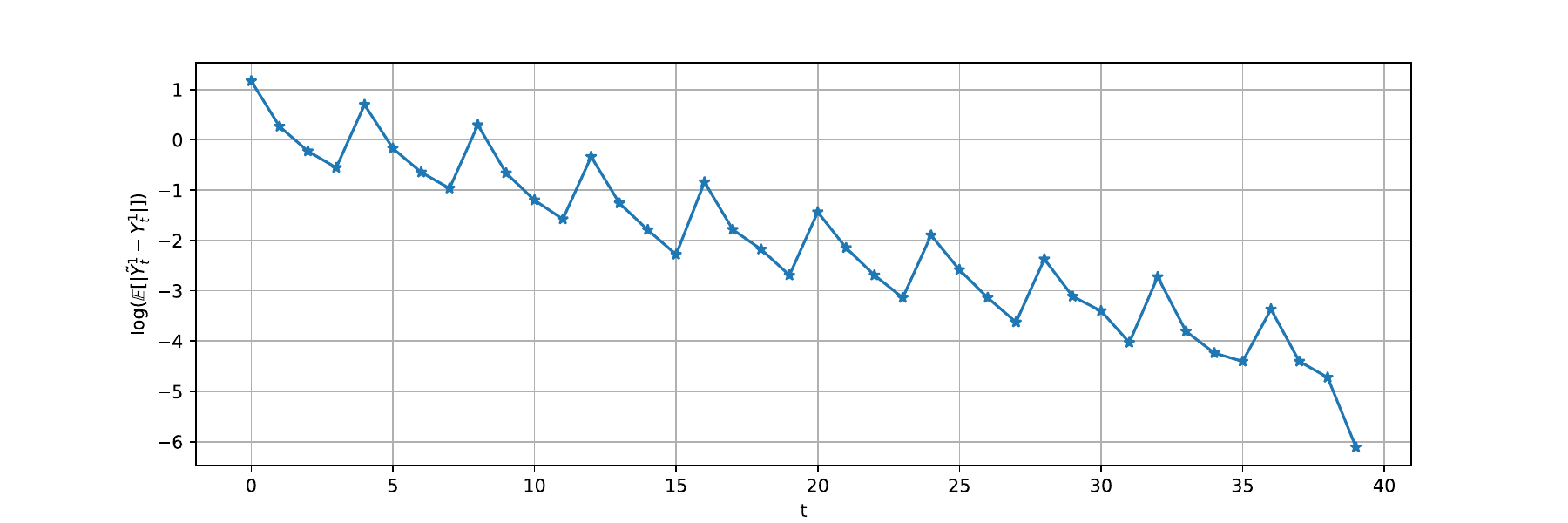}
        \caption{The linear decay of $\log \E |\tilde Y_t - Y_t|$. The empty history count process $Y$ joins its periodically stationary version $\tilde Y$ exponentially fast. The expected value is approximated by averaging $N_{MC}=2500$ trajectories.}
        \label{fig:decay}
    \end{figure}

 \section{The universal Markov approximation of Poisson autoregressions} 
 \label{sec:approx}
 In this section, we propose a universal Markov approximation of the periodic multivariate Poisson autoregression that is parametrically more parsimonious; this is particularly interesting for our model. For instance, if the regression involves $m$ lags in the past, then the total number of coefficients of the regression matrices is $d^2 m p$, which can be too high. In the network setting (\cite{ArFok,knight2016modellingdetrendingdecorrelationnetwork}) we can reduce the number of interaction terms leading to a more parsimonious model. We seek here to find a way to reduce the lag $m$ in case the process has a long or infinite memory. We first prove continuity results on the kernel, which means that for two autoregression kernels that are close enough in some metric, the two resulting Poisson autoregressions are close.
 \subsection{Continuity of the Poisson autoregression with respect to the kernels}
Let $(\phi^{(t)}_k)_{t\in \Z, k \in \N^*}$ and $(\bar \phi^{(t)}_k)_{t\in \Z, k \in \N^*}$ be two matrix kernels such that 
\begin{equation}
\label{ineq:dom}
\max(|\phi^{(t)}_k|,|\bar \phi^{(t)}_k|) \preceq A_k, \text{ for any }t\in \{1,\cdots,p\} \text{ and }k\in \N^*,
\end{equation}
for a family of non-negative matrices satisfying the stability condition $\rho \left( L\sum_{k\geq 1}A_k\right)<1$. \\
We build two periodic multivariate Poisson autoregressions (of either Type I or Type II) $(Y)_{t\in \N}$ and $(\bar Y_t)_{t\in \N}$ using kernels $\phi$ and $\bar \phi$, respectively. The two processes are constructed using the same underlying randomness, that is the same Poisson processes $(N_t)_{t\in \Z}$. We now give a control on the distance between $Y$ and $\bar Y$ as a function of the distance between $\phi$ and $\bar \phi$. 
\begin{Proposition}
    \label{prop:continuity}
    Let $\psi$ be an $L-$Lipschitz non-negative function and $(\mu_t)_{t\in \N}$ be a periodic family of vectors in $\R^d$. 
    Let $\phi$ and $\bar \phi$ be two kernels satisfying (\ref{ineq:dom}) for a family of matrices $(A_k)_{k \in \N^*}$ such that $\rho \left (L\sum_{k\geq 1} A_k \right)<1$.\\
    Given a family $N$ of \textit{iid} unit intensity Poisson processes:
    \begin{enumerate}
        \item If $Y$ (resp. $\bar Y$) is constructed according to equation \eqref{def:TypeI} using kernel $\phi$ (resp. $\bar \phi$), then for any $t \in \N$ and any $r\geq 1$
        $$\E^{1/r} \left[ |Y_t-\bar Y_t|^r\right] \preceq C\left(\max_{v=1,\cdots,p}\|\phi^{(v)}-\bar \phi^{(v)}\|_1\right)^{1/r} \mathds 1.$$
        \item If $Y$ (resp. $\bar Y$) is constructed according to equation \eqref{def:TypeII} using kernel $\phi$ (resp. $\bar \phi$), then for any $t \in \N$ and any $r \geq 1$
        $$\E^{1/r} \left[ |Y_t-\bar Y_t|\right] \preceq C \left(\sum_{v=1}^p\|\phi^{(v)}-\bar \phi^{(v)}\|_1 \right)^{1/r}\mathds 1,$$
    \end{enumerate}
    where $C$ is a positive constant that does not depend on $t$, $\|\phi^{(v)}-\bar \phi^{(v)}\|_1=\sum_{k\geq 1}|\phi^{(v)}_k-\bar \phi^{(v)}_k|$ and $\mathds 1$ is the $d-$dimensional vector whose coefficients are $1$.
\end{Proposition}
\begin{proof}
The proof can be found in Appendix \ref{sec:prop:continuity}.
\end{proof}
We point out that the constant vector $C$ can be explicitly expressed in terms of $\left(I-L\sum_{k\geq} A_k\right)$, $\psi$ and $\max_{v=1,\cdots,p} \mu_v$. We also point out that the result still holds for the strictly periodically stationary processes. We now propose the Poisson autoregression with an exponential polynomial as a universal approximation, but first, we study some of its properties. 
\subsection{The Markov properties of the Poisson autoregression with an exponential polynomial kernel}
\label{section:markov}
To fix the ideas, we consider a linear Poisson autoregression with a constant baseline intensity and kernel, that is 
\begin{equation}
    \label{eq:constant}
    \begin{cases}    
        Y_t&=N_t \left( \lambda_t\right)\\
        \lambda_t&=\mu + \sum_{k=1}^{t-1}\phi_{t-k} Y_k,
    \end{cases}
\end{equation}
where $\mu \in \R_+^d$ and $\phi$ is a family of non-negative matrices such that $\rho(\sum_{k\geq 1}\phi_k)<1$. Clearly, $(Y_t)_{t\in \N}$ is not a Markov chain in general, as the distribution of $Y_t$ depends on the entire history until $t-1$. This is problematic for instance for the computation of the likelihood, which for a sequence of observation of length $T$ would cost $O(T^2)$ operations. A Markov approximation proposed in \cite{DFT} based on the results proven in \cite{DW} would be the sequence $(Y^{(m)}_t,\cdots, Y^{(m)}_{t-m})_{t\in \N}$, where  
\begin{equation*}
    \begin{cases}
        Y^{(m)}_t&=N_t \left( \lambda^{(m)}_t\right)\\
        \lambda^{(m)}_t&=\mu + \sum_{k=t-m}^{t-1}\phi_{t-k} Y^{(m)}_k.
    \end{cases}
    \end{equation*}
    The Markov property helps reduce the cost of computation of the likelihood to $O(mT)$.
    As we saw in Proposition \ref{prop:continuity}, to guarantee that such an approximation is good, one must ensure that the remainder $\sum_{k>m}\phi_k$ is small. In case the autoregression has a long memory, we would need a large $m$, that can be of order $T$, which does not lead to a significant reduction in the cost of computation.
    
    The alternative Markov approximation is based on the following observation: If the kernel is of the form $$\phi_k=G e^{-\frac{k}{\tau}},$$
    where $G$ is a non-negative matrix and $\tau>0$, then the intensity of \eqref{eq:constant} can be expressed as
    \begin{align*}
        \lambda_t&=\mu + \sum_{k=1}^{t-1}G e^{-\frac{t-k}{\tau}} Y_k\\
        &=\mu + e^{-\frac{1}{\tau}}\left(\sum_{k=1}^{t-2}G e^{-\frac{t-1-k}{\tau}} Y_k + G Y_{t-1}\right)\\
        &=(1-e^{-\frac{1}{\tau}})\mu + e^{-\frac{1}{\tau}} \lambda_{t-1} + e^{-\frac{1}{\tau}}G Y_{t-1}.
    \end{align*}
Hence, if the kernel is a geometric sequence, then the vector $(Y,\lambda)$ is a standard Markov chain. This is exactly the linear multivariate count autoregression model studied in \cite{FSTD}, and its stationarity and ergodicity are proven using standard Markov techniques on a perturbed approximation of the chain. This process is also a multivariate version of the INGARCH $(1,1)$ model \cite{FLO}. In fact, we have just proven that the INGARCH $(1,1)$ process is identical to the INGARCH $(+\infty,0)$ with exponential regression coefficients. 
\begin{Remark}
\label{remark:abd}
    The linear periodic INGARCH$(1,1)$ model introduced in \cite{ABD} is equivalent to the Type II periodic Poisson autoregression with an exponential kernel. Indeed, if $\phi^{(t)}_k=\nu_{t+1} e^{-\frac{k}{\tau}}$ then 
    \begin{align*}
        \lambda_t= &\mu+\sum_{k=1}^{t-1}\nu_{k+1}e^{-\frac{t-k}{\tau}}Y_k\\
        =&\mu+ \left(e^{-\frac{1}{\tau}}\sum_{k=1}^{t-2}\nu_{k+1}e^{-\frac{t-1-k}{\tau}}Y_k + \nu_te^{-\frac{1}{\tau}}Y_{t-1}\right)\\
        =&\mu (1-e^{-\frac{1}{\tau}}) + e^{-\frac{1}{\tau}}\lambda_{t-1} + \nu_t e^{-\frac{1}{\tau}}Y_{t-1}.
    \end{align*}
\end{Remark}

We now show that the exponential polynomial kernels of the form
$$\phi_k=\sum_{m=1}^{q} G^{(m)} e^{-m\frac{k}{\tau}},$$
where $(G^{(m)})_{m=1,\cdots,q}$ is a family of matrices, also imply that the time series are Markov, up to adding some auxiliary processes. In this case, the intensity is of the form 
\begin{align*}
    \lambda_t&=\mu + \sum_{k=1}^{t-1} \sum_{m=1}^{q} G^{(m)} e^{-m\frac{t-k}{\tau}} Y_k\\
    &=\mu +  \sum_{m=1}^{q} G^{(m)} \sum_{k=1}^{t-1}e^{-m\frac{t-k}{\tau}} Y_k\\
    &=\mu +  \sum_{m=1}^{q} G^{(m)}\xi^{(m)}_t,
\end{align*}
where $\xi_t^{(m)}:=\sum_{k=1}^{t-1}e^{-m\frac{t-k}{\tau}} Y_k$ is the $m$-th auxiliary process. As before, we have that
$$\xi^{(m)}_t=e^{-\frac{m}{\tau}}\left (\xi^{(m)}_{t-1}+Y_{t-1}\right),$$
for $m\in \{1,\cdots,q\}$ and $t>1$. Hence, \eqref{eq:constant} is equivalent to 
\begin{equation*}
    \begin{cases}
    Y_t&=N_t \left(\mu+ \sum_{m=1}^{q} G^{(m)}\xi^{(m)}_t\right)\\
    \xi^{(m)}_t&=e^{-\frac{m}{\tau}} \xi^{(m)}_{t-1} + e^{-\frac{m}{\tau}} Y_{t-1}, \quad m\in \{1,\cdots,q\},
    \end{cases}
\end{equation*}
clearly implying that $(Y, \xi^{(1)},\cdots,\xi^{(q)})$ is a Markov chain. One of the advantages of the Markov property is that the cost of simulation and likelihood computation for a sample of length $T$ is of order $O(qT)$, which in the case of $q<<T$ leads to a significant reduction in computation time.  We now state our universal approximation result for periodic Poisson autoregressions with Markov chains.
\begin{Theorem}
\label{thm:4.3}
    Let $(\phi^{(t)}_k)_{t\in \Z, k\in \N^*}$ be a family of matrix kernels satisfying the stability Assumption \ref{ass:stability}. Let $\tau>0$ be a fixed characteristic time and let $\varepsilon >0$.
    \begin{enumerate}
        \item If $Y$ is a multivariate Poisson autoregression of Type I periodicity given in \eqref{def:TypeI}, then there exists $q \in \N^*$ and a family of matrices $(G_t^{(m)})_{t \in \Z, m=1,\cdots,q}$ periodic in $t$ such that 
        $$\E [|Y_t-\bar Y_t|] \preceq \varepsilon \mathds 1,$$
        where $(\bar Y,\xi^{(1)},\cdots,\xi^{(q)})$ is the Markov chain
        \begin{equation*}
            \begin{cases}
                \bar Y _t &= N_t \left(\psi\left( \mu_t + \sum_{m=1}^{q} G_t^{(m)} \xi^{(m)}_t\right)\right)\\
                \xi^{(m)}_t &= e^{-\frac{2m+1}{\tau}} \xi^{(m)}_{t-1} + e^{-\frac{2m+1}{\tau}} \bar Y_{t-1}, \quad m\in \{1,\cdots,q\}.
            \end{cases}
        \end{equation*}
        \item 
        If $Y$ is a multivariate Poisson autoregression of Type II periodicity given in \eqref{def:TypeII}, then there exists $q \in \N^*$ and a family of matrices $(J_t^{(m)})_{t\in \Z, m=1,\cdots,q}$ periodic in $t$ such that 
        $$\E [|Y_t-\bar Y_t|] \preceq \varepsilon \mathds 1,$$
        where $(\bar Y,\zeta^{(1)},\dots,\zeta^{(q)})$ is the Markov chain
        \begin{equation*}
            \begin{cases}
                \bar Y _t &= N_t \left( \psi \left( \mu_t + \sum_{m=1}^{q}  \zeta^{(m)}_t\right) \right)\\
                \zeta^{(m)}_t &= e^{-\frac{2m+1}{\tau}} \zeta^{(m)}_{t-1} + e^{-\frac{2m+1}{\tau}} J_{t-1}^{(m)} \bar Y_{t-1}, \quad m\in \{1,\cdots,q\}.
            \end{cases}
        \end{equation*}
    \end{enumerate}
\end{Theorem}
The proof can be found in Appendix \ref{proof:4.3}.

\begin{Remark}
    In this paper, we approximate a given kernel family $(\phi^{(t)}_k)_{k \in \N^*}$ using a linear combination of odd powers of the exponential $\sum_{m=1}^q G^{(m)}_t e^{-(2m+1)\frac{k}{\tau}}$ because of the proof of Lemma \ref{lmm:l1_desity}. However, it is also possible to use a linear combination of the exponential $\sum_{m=1}^q G^{(m)}_t e^{-(m)\frac{k}{\tau}}$ as an approximation as well.
\end{Remark}

To illustrate the universality of the approximation, we simulate a Type I periodic linear univariate Poisson autoregression of kernel 
$$\phi^{(t)}_k=\frac{(\mathds 1_{t=2\text{ mod}[4]}+\mathds 1_{t=3\text{ mod}[4]})k^{1.5}}{75(1+(0.2 \cdot k )^{3.5})},$$
which has a longer memory than the finite kernels or the exponential kernel due to the tail that vanishes like $k^{-2}$.

Figure \ref{fig:approx_kernel} shows $\phi^2$ (in blue). As we can see, the lag coefficients $\{3,\cdots,7\}$ have a significantly higher impact than the first two lag coefficients. Such an effect is impossible to capture with the exponential kernel (equivalently INGARCH $(1,1)$ extensively studied in the literature). However, in real-life applications, the exponential kernel is not always the one that best fits the data. For instance, for Hawkes processes, that are continuous-time versions of Poisson autoregressions \cite{mahmoud}, it has been shown by Bessy-Roland \textit{et al.} \cite{Bessy-Roland_Boumezoued_Hillairet_2021} that a cyber attack model driven by a Hawkes process with a kernel $\phi(s)=\alpha s e^{-\beta s}$ fits the data better than a Hawkes process driven by a simple exponential. 
\begin{figure}[h!]
    \centering
    \includegraphics[width=1\linewidth]{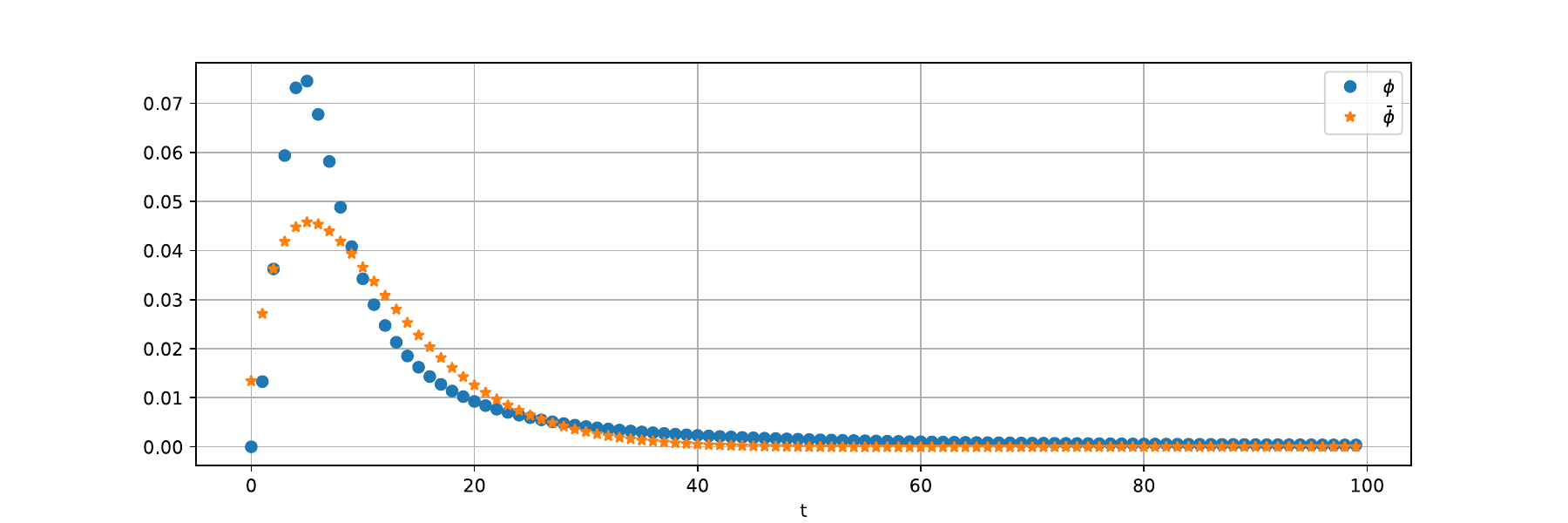}
    \caption{The $\ell_1(\N)$ approximation of $\phi^2$ using $q=3$ exponential functions with $\tau=36$, in orange.}
    \label{fig:approx_kernel}
\end{figure}
Before proceeding to the Poisson autoregressions, we state a couple of remarks on the $\ell_1(\N^*)$ approximation of kernels:
\begin{Remark}
\label{remark:tau}
\begin{enumerate}
    \item While in theory the exponential polynomials are dense in $\ell_1(\N^*)$ for any $\tau>0$, some values of $\tau$ will perform better than others for a fixed number of exponentials $q$. Throughout this paper we use this method to select $\tau$:  We fix $T_c$ as a cutoff time after which the past becomes negligible. We pick $\tau=\frac{3}{5}T_c$ which ensures that all of the exponentials are below $e^{-5}$ at $t=T_c$. In Figure \ref{fig:approx_kernel}, we chose $T_c=60$ yielding $\tau=36$.
    \item Computing the best $\ell_1(\N^*)$ approximation can be done either by inverting the system that emerges from setting the gradient of the $\ell_2(\N^*)$ distance to zero (\textit{cf.} Lemma \ref{lmm:l2_density}), or directly numerically. The approximation on Figure \ref{fig:approx_kernel} is obtained by numerically minimising the $\ell_1(\N^*)$ using the method \texttt{COBYLA} in \texttt{scipy.optimize.minimize}.
\end{enumerate}
\end{Remark}
In Figure \ref{fig:approx} we simulate a Type I periodic Poisson autoregression with kernel $\phi$ (in blue) as well as a trajectory with the same underlying randomness $N$ with the kernel $\bar \phi$ that is the best approximation of $\phi$ in $\ell_1(\N^*)$ with $q=3$ exponentials (in orange) and using the truncated kernel $\hat \phi^{(t)}_k=\phi^{(t)}_k \mathds 1_{k\leq 3}$ (in green). The approximation of Poisson autoregression of infinite order with finite order Markov chains using truncation has been proposed in \cite{DFT}. We notice that while both $\bar Y$ and $\hat Y$ cost $O(3T)$ operations to simulate (compared to $O(T^2)$ for $Y$), $\bar Y$ approximates $Y$ significantly better than $\hat Y$. For instance, the cumulative count error for exponential polynomial approximation is $\sum_{k=1}^{120}|Y_k-\bar Y_k|=9$, whereas for the truncated kernel $\sum_{k=1}^{120}|Y_k-\hat Y_k|=39$. 

\begin{figure}[h!]
    \centering
    \includegraphics[width=1\linewidth]{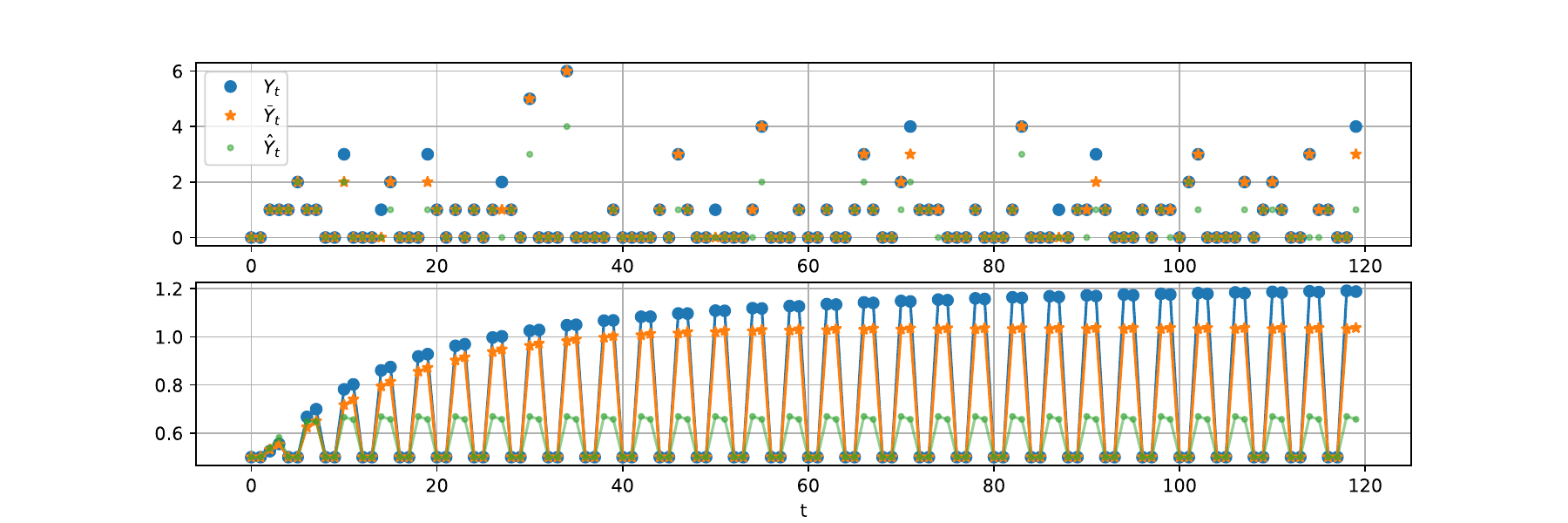}
    \caption{Upper: Simulation of the Poisson autoregression with the original kernel in blue, as well as its approximation using an exponential polynomial of order $q=3$ in orange and truncating at $3$ regression coefficients. Lower: $\E[Y_t], \E[\bar Y_t]$ and $\E[\hat Y_t]$ in blue, orange and green respectively. Note how $\E[Y_t]$ reaches its periodically-stationary regime slower than $\E[\bar Y_t]$ due to its heavier tail.}
    \label{fig:approx}
\end{figure}

\begin{Remark}
\label{Remark:parsimonious}
    To make the model more parsimonious, it is possible to parametrise the periodically varying coefficients by a linear combination of sine and cosine functions 
    $$G^{(m)}_t=\sum_{j=1}^r \gamma^{(m)}_j + \gamma_j^{(m)'}\sin\left(\frac{2 \pi t}{p}\right)+ \gamma_j^{(m)''}  \cos \left(\frac{2 \pi t}{p}\right).$$
    This reparametrisation is particularly interesting if $3r <<p$. One can, for instance, think of a model of daily count data  with a yearly seasonality. In this case, $p=365$ which can be significantly reduced if the bulk of each coefficient can be captured by a few trigonometric functions. 
\end{Remark}
\section{Statistical inference}
\label{sec:inference}
The estimation of non-periodic multivariate count series has been studied in the literature using different methods: Likelihood maximisation (multivariate INGARCH$(1,1)$ models \cite{FSTD,SKK}, PNAR$(p)$ process \cite{ArFok}) conditional least squares (multivariate INAR$(p)$ \cite{Latour,Kirchner_multi,EDD}, we point that the last two view the count series as a proxy for Hawkes processes) and expectation maximisation \cite{SCN}, which has been shown to perform better than the former two methods in case the kernel is exponential. \\
In this section, we study the parametric inference of the periodic Poisson autoregression with an exponential polynomial kernel using the maximum likelihood estimation (MLE). Given a sample of observations $\{Y_t\}_{t=1,\cdots,pT}$ of $d-$variate integer-valued vectors, the Poisson log-likelihood can be written  as 
\begin{align}
    \label{eq:likelihood}
    L _{T}(\theta) &= \frac{1}{T}\sum_{t=1}^{pT}\sum_{i=1}^d Y^{(i)}_t \log \lambda_t^{(i)}(\theta) - \lambda_t^{(i)}(\theta)\\
    &=\frac{1}{T}\sum_{t=1}^{pT}\sum_{i=1}^d\ell_t^{(i)}(\theta),  \nonumber
\end{align}
where $\theta$ is a parameter to be estimated, see for instance \cite{FSTD} or \cite{SKK}. In our case, we would like to estimate the baseline pre-intensities $\mu_v$ as well as the kernels $(\phi^{(v)}_k)_{k\in \N^*}$ for $v=1,\cdots,p$, giving
$$\theta=(\mu_1,\cdots,\mu_v,(\phi^{(1)}_k)_{k\in \N^*},\cdots,(\phi^{(v)}_k)_{k\in \N^*}),$$
which is an infinite dimensional parameter.
We focus on the Type I periodic Poisson autoregressions with an exponential polynomial kernel, that is 
\begin{equation*}
    \begin{cases}
        Y_t&=N_t\left(\lambda_t\right),\\
        \lambda_t &= \psi \left(\mu_t+\sum_{k=1}^{t-1}\left(\sum_{m=1}^{q} G_t^{(m)} e^{-m\frac{t-k}{\tau}}\right)Y_k\right)\\
        &=\psi \left(\mu_t+\sum_{k=1}^{t-1}\left(\sum_{m=1}^{q} G_t^{(m)} e^{-m\frac{k}{\tau}}\right)Y_{t-k}\right)
    \end{cases}
        \text {for } t \in \N,
\end{equation*}
for a family of matrices $G_t^{(m)}$ periodic in $t$ and a fixed time parameter $\tau>0$. We point out that the Type II periodicity can be treated in a similar way, \textit{mutatis mutandis}.
\subsection{Properties of the Markov maximum likelihood estimator}
In Section \ref{section:markov} we showed that the Poisson autoregression with exponential polynomial kernels can be represented as the Markov chain of order $q$ 
\begin{equation}
\label{eq:markov_MLE}
    \begin{cases}
        Y_t&=N_t\left(\psi \left(\mu_t+\sum_{m=1}^{q} G^{(m)}_t\xi^{(m)}_t\right)\right)\\
        \xi^{(m)}_t&=e^{-\frac{2m+1}{\tau}}\xi^{(m)}_{t-1}+e^{-\frac{2m+1}{\tau}}Y_{t-1}, \quad m\in \{1,\cdots,q\}.
    \end{cases}
\end{equation}
Throughout this section, $\tilde Y$  and $\tilde \xi$ denote the counts and the auxiliary process in the periodically stationary regime. $Y$ and $\xi$ denote the solutions with a fixed history, usually taken to be empty.\\ 
The likelihood in the periodically stationary regime is denoted by $\tilde L$ and the likelihood of the solution with a fixed history is denoted by $L$.

We remind the reader that due to the density of exponential polynomials in $\ell_1(\N^*)$ and the continuity of Poisson autoregressions with respect to the kernel, we will assume that the observed counts come from a Poisson autoregression with an exponential polynomial kernel. This will introduce a misspecification error in the estimation that we will examine empirically. To the best of our knowledge, Douc \textit{et al.} \cite{DDM} as well as Armillotta \textit{et al.} \cite{ALL} proved some results for misspecified MLE for the observation-driven (closely related to INGARCH(1,1)) count series, but the case of general Poisson autoregressions with infinite memory is yet to be thoroughly studied.

In the well-specified case, we seek to estimate the baseline intensities $(\mu_v)_{v=1,\cdots,p}$ as well as the regression matrices $(G^{(m)}_v)_{v=1,\cdots,p}$. The jump rate function $\psi$ is supposed to be known, as well as the order $q\in \N^*$ and the time parameter $\tau>0$. The reason for which we do not seek to estimate $\tau$ stems from the density of exponential polynomials in $\ell_1(\N^*)$ for \textit{any} given $\tau>0$. 

For any $v\in \{1,\cdots,p\}$, the kernel $(\phi^{(v)}_k)_{k\in \N^*}$ will be estimated via its parameters $(G^{(m)}_v)_{m=1,\cdots,q}$. This will take us from an optimisation problem over the infinite dimensional space $(\ell_1(\N^*))^{pd^2}$ to an optimisation problem over a finite dimensional space.

Following the lines of \cite{almohaimeed}, the parameter vector for the MLE associated with Equation \eqref{eq:markov_MLE} is formed by the concatenation of the parameters across all seasons, that is 
\begin{equation}
    \label{eq:theta}
    \gamma^*=\left(\mu_1,\cdots,\mu_p,G^{(1)}_1,\cdots,G^{(q)}_1,\cdots,G^{(1)}_2,\cdots,G^{(q)}_2,\cdots, G^{(1)}_p,\cdots,G^{(q)}_p\right) \in \Gamma \subset \R^{p(1+d^2q)},
\end{equation}
$\Gamma$ here being a compact parameter space. 
For this process, the reparametrised Markov log-likelihood is of the form 
\begin{equation}
\label{eq:likelihood_markov}
\begin{cases}
    L_{T}(\gamma) =& \sum_{t=1}^{pT}\sum_{i=1}^d \tilde Y^{(i)}_t\log \left(\psi \left(\mu^{(i)}_t + \left(\sum_{m=1}^{q}G^{(m)}_t\xi^{(m)}_t\right)^{(i)}\right)\right)\\
    &-\psi \left(\mu^{(i)}_t + \left(\sum_{m=1}^{q}G^{(m)}_t\xi^{(m)}_t\right)^{(i)}\right)\\
    \xi^{(m)}_t=&e^{-\frac{2m+1}{\tau}}\left(\xi^{(m)}_{t-1}+\tilde Y_{t-1}\right), \quad m\in\{1,\cdots,q\}\\
    \xi^{(m)}_0=&0,
    \end{cases}
\end{equation}
and can be computed in a time of order $O(qT)$, as opposed to $O(T^2)$ for a Poisson autoregression with a general kernel.\\
The MLE is then defined as any measurable solution of 
\begin{equation}
\label{eq:def_likelihood}
\gamma_T = \text{argmax}_{\gamma \in \Gamma} L_{T}(\gamma),
\end{equation}
where $L$ is the log likelihood defined in Equation \eqref{eq:likelihood_markov} with an empty history. As we showed in Section \ref{section:stability}, the initial state of the process is asymptotically irrelevant provided the stability assumption is met. We now state the assumptions for the strong consistency of the MLE.

\begin{Assumption}
\label{ass:mle}
    \begin{enumerate}
        \item Stationarity: \eqref{eq:markov_MLE} saisfies the stability Assumption \ref{ass:stability} and has a periodically stationary and ergodic solution $\tilde Y$. 
        \item Positivity: There exists  $\varepsilon >0$ such that $\psi(x)\geq \varepsilon$ for all $x \in\R $.
        \item Compactness: $\Gamma$ is a compact subset of $\R^{p(1+d^2q)}$ and contains the true parameter $\gamma^*$.
        \item Identifiability: If for $\gamma$ and $\gamma'$ we have $\tilde \lambda_v(\gamma)=\tilde \lambda_v(\gamma')$ for $v\in \{1,\cdots,p\},$ then $\gamma=\gamma'$.
    \end{enumerate}
\end{Assumption}
Building on the seminal work of Ahmad and Francq \cite{AF}, we now prove the strong consistency of the MLE. 
\begin{Theorem}
\label{thm:strong-cosistency}
    Assume that $\tilde Y$ is the unique periodically stationary solution of \eqref{eq:markov_MLE} and that Assumption \ref{ass:mle} is in force. Then the MLE estimator defined by \eqref{eq:def_likelihood} satisfies 
    $$\lim_{T\to +\infty}\gamma_T = \gamma^*, \quad \text{almost surely,}$$
    and therefore, if we consider the reconstructed kernels $\phi^{(v)}_{k,T}=\sum_{m=1}^{q} G^{(m)}_{T,v} e^{-m\frac{k}{\tau}}$ for any $v=1,\cdots,p$ we have 
    $$\lim_{T\to +\infty}\mu_{v,T}=\mu_v \quad \text {and}\quad \lim_{T\to +\infty}\phi_{T}^{(v)}=\phi^{(v)} \text{ in } \ell_1(\N^*)$$
    almost surely.
\end{Theorem}
\begin{proof}
The proof can be found in Section \ref{sec:thm:strong-cosistency}
\end{proof}
Type I periodicity allows for a factorisation of the log-likelihood, that is the possibility of writing $L_T(\theta)$ as a sum of seasonal log-likelihoods $L^{(v)}_T(\theta_v)$ that depend only on the seasonal parameter $\theta_v$. Indeed, by a change of counter in the sum in \eqref{eq:likelihood_markov}, we have 
\begin{align*}
    L_T(\gamma)=&\frac{1}{T}\sum_{n=1}^{T-1}\sum_{v=1}^p\sum_{i=1}^d \tilde Y^{(i)}_{np+v}\log \left(\psi \left(\mu^{(i)}_{np+v} + \left(\sum_{m=1}^{q}G^{(m)}_{np+v}\xi^{(m)}_{np+v}\right)^{(i)}\right)\right)\\
    &-\psi \left(\mu^{(i)}_{np+v} + \left(\sum_{m=1}^{q}G^{(m)}_{np+v}\xi^{(m)}_{np+v}\right)^{(i)}\right)\\
    =&\frac{1}{T}\sum_{v=1}^p \sum_{n=1}^{T-1}\sum_{i=1}^d \tilde Y^{(i)}_{np+v}\log \left(\psi \left(\mu^{(i)}_{v} + \left(\sum_{m=1}^{q}G^{(m)}_{v}\xi^{(m)}_{np+v}\right)^{(i)}\right)\right)\\
    &-\psi \left(\mu^{(i)}_{v} + \left(\sum_{m=1}^{q}G^{(m)}_{v}\xi^{(m)}_{np+v}\right)^{(i)}\right)\\
    =&\sum_{v=1}^p L^{(v)}_T(\gamma_v),
\end{align*}
where $\gamma_v=(\mu_v,G^{(1)}_v,\cdots,G^{(q)}_v)$. This means that the maximisation of $L$ can be done by maximising each of the seasonal likelihoods 
$$\gamma_{v,T}=\text{argmax}_{\gamma_v \in \Gamma_v} L^{(v)}_T(\gamma_v), \quad v=1,\cdots,p,$$
which is numerically more efficient than maximising $L$. This is the method we use for the numerical illustrations. 

To conclude this subsection, we point out that the likelihoods $L^{(v)}_T$ have another helpful characteristic, beyond their linear computational cost: If the jump-rate function $\psi$ is convex and its logarithm is concave, then $L^{(v)}_T$ is concave and thus has exactly one maximum on compacts. This is the case for instance for the linear autoregression $\phi(x)=x$ or for the softplus $\psi(x)=\ln (1+e^x)$.

\subsection{Numerical illustrations}
Throughout this subsection, we consider a Type I periodic autoregression on a $12$-node network, that is 
$$\lambda_t=\psi \left(\mu_t+\sum_{k=1}^{t-1}(\alpha^{(t)}_{t-k}I_{12}+\beta^{(t)}_{t-k}W)Y_k\right),$$
where $W$ is the normalised adjacency matrix, generated using the stochastic block model with two blocks.

The jump rate is chosen to be $\psi(x)=\ln \left(1+e^x\right)+0.01$. This choice ensures that $\psi$ is Lipschitz continuous and bijective, hence ensuring identifiability in Assumption \ref{ass:mle}. The offset $0.01$ is there to guarantee the positivity, however, it comes at the price of the log-likelihood's global concavity. The goal is to estimate, from data, the vector 
$$\theta^*_v =(\mu_v,(\alpha^{(v)}_k)_{k\in \N^*}, (\beta^{(v)}_k)_{k\in \N^*}),$$
for all seasons $v=1,\cdots,p$. The momentum kernel $\alpha$ and the network kernel $\beta$ will be estimated parametrically, assuming that they are exponential polynomials:
$$\alpha_k^{(v)}=\sum_{m=1}^{q} a^{(v)}_m e^{-\frac{(2m+1)}{\tau}k},$$
and 
$$\beta_k^{(v)}=\sum_{m=1}^{q} b^{(v)}_m e^{-\frac{(2m+1)}{\tau}k}.$$
Of course, if the data comes from kernels that are not exponential polynomials, then this will introduce a misspecification error. We examine it here, both in the case of light-tailed (at least exponentially fast decay) and heavy-tailed (slower than exponential decay) kernels. 

For all the following simulations, the Markov likelihood is maximised using the method \texttt{BFGS} in \texttt{scipy.optimize.minimize}. For the initial values, we take $a^{(v)}_m=b^{(v)}_m=0$ for all $v=1,\cdots,p$ and $m=1,\cdots,q$ and $\mu^{(v)}=\psi^{-1}\left(\frac{1}{Tpd}\sum_{k=1}^{Tp}\sum_{j=1}^d Y_k^{(j)}\right)$. The period is set to $p=7$ and for the baseline pre-intensity we take 
$\mu_v= \mathds 1_{v\leq 3}.$
\subsubsection{Estimation for the well specified model}
The kernels are chosen as
$$\alpha^{(v)}_k=\left(e^{-\frac{3}{4}k}+0.5e^{-\frac{5}{4}k}-1.5e^{-\frac{7}{4}k}-2e^{-\frac{9}{4}k}\right)\frac{1+\cos(2\pi v/7)}{2},$$
and 
$$\beta^{(v)}_k=\left(1.5e^{-\frac{3}{4}k}+1.5e^{-\frac{5}{4}k}-4e^{-\frac{7}{4}k}-5e^{-\frac{9}{4}k}\right)\sin(2\pi v/7),$$
giving $(a^{(v)}_1,a^{(v)}_2,a^{(v)}_3,a^{(v)}_4)=(1,\frac{1}{2},-\frac{3}{2},-2)\cdot\frac{1}{2}\left(1+\cos(2\pi v/7)\right)$ and $(b^{(v)}_1,b^{(v)}_2,b^{(v)}_3,b^{(v)}_4) = (\frac{3}{2},\frac{3}{2},-4,-5)\cdot\sin(2\pi v/7)$.
The correct order $q=4$ and characteristic time $\tau=4$ are assumed known. The aforementioned maximum likelihood estimator is applied for $T=200$ periods, over $N_{MC}=40$ simulations, each time yielding an estimate $(\hat \mu_v,\hat a_1^{(v)},\cdots, \hat a_4^{(v)},\hat b_1^{(v)},\cdots, \hat b_4^{(v)})$. The results for $\mu$ are reported in Figure \ref{fig:boxplot_mu}.
\begin{figure}[h!]
    \centering
    \includegraphics[width=0.7\linewidth]{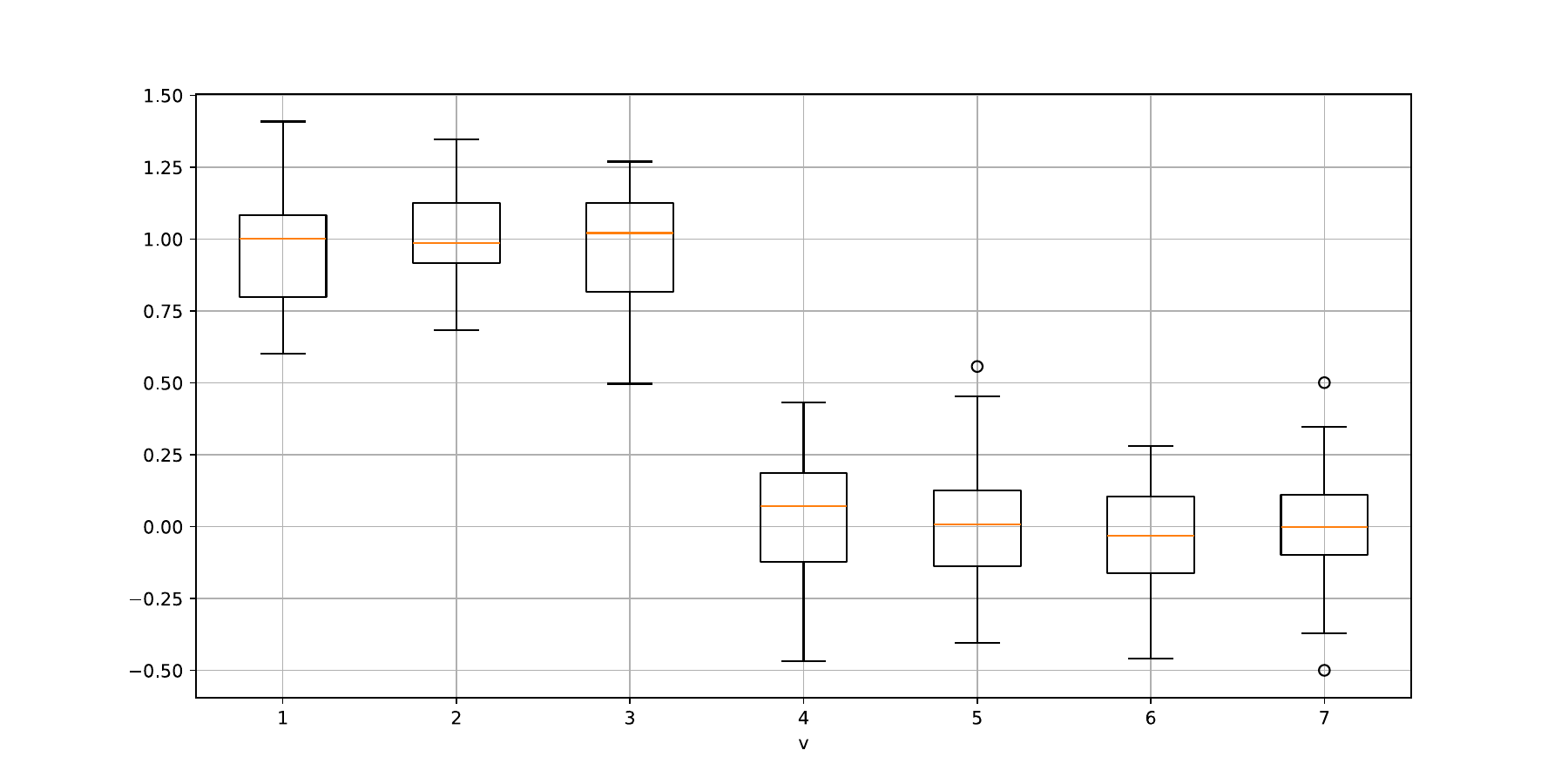}
    \caption{A boxplot of the baseline pre-intensities. The average values are $(0.97,  1.00,  0.97,  0.03,  0.00,
       -0.06 , 0.00)$ and the standard deviations are $(0.19, 0.14, 0.21, 0.22, 0.21, 0.19, 0.19)$.}
    \label{fig:boxplot_mu}
\end{figure}
For the kernels, given the $N_{MC}=40$ simulations, we plot in Figure \ref{fig:kernels_a_b_ws} the reconstructed $\hat \alpha$ and $\hat \beta$ from the estimated  coefficients $\hat a$ and $\hat b$. We also plot the reconstructed kernels using the average coefficients $N_{MC}^{-1} \sum_{n=1}^{N_{MC}} (\hat a_m^{(v)})_n$ and  $N_{MC}^{-1} \sum_{n=1}^{N_{MC}} (\hat b_m^{(v)})_n$.

In the well-specified case, we notice that the average reconstructed kernel fits the ground truth trajectory very well. Furthermore, the trajectory that is reconstructed at every simulation captures the overall appearance of $\alpha$ and $\beta$: almost every green curve has the same overall effect (excitation/inhibition) at a similar order of magnitude as the corresponding blue trajectory. The non-instantaneous excitation/inhibition seems to be well captured as well.

\begin{figure}
    \centering
    \includegraphics[width=1\linewidth]{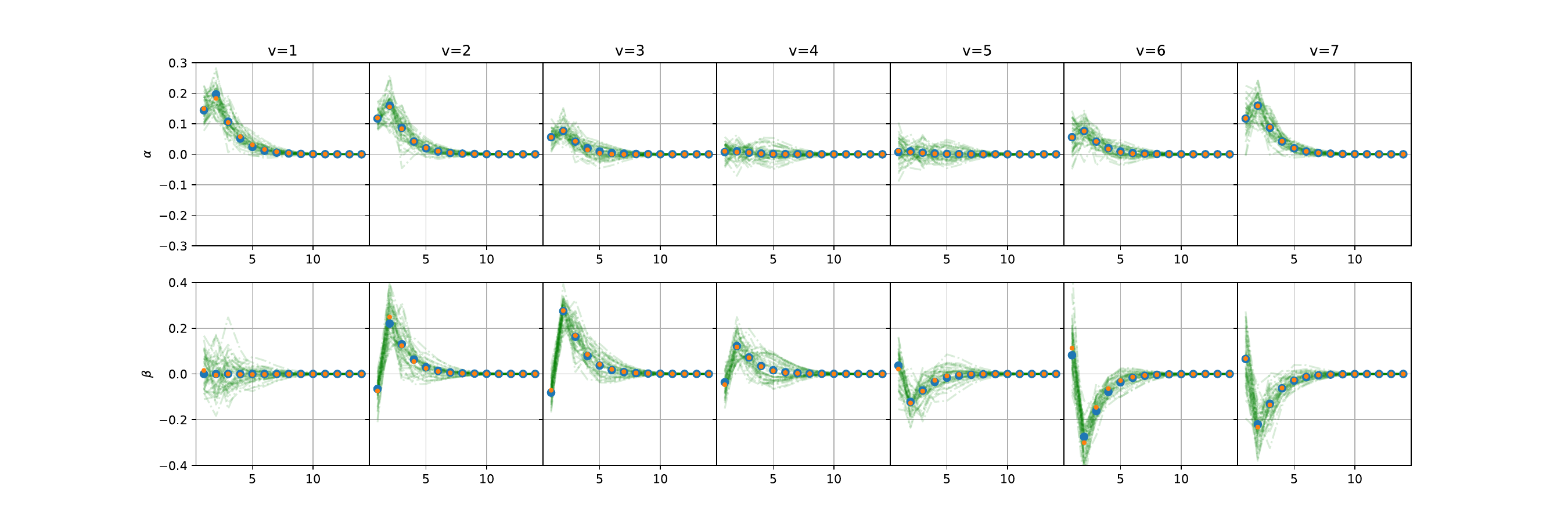}
    \caption{The MLE results in the well specified case, with the known characteristic time $\tau=4$ and order $=4$. In blue: the ground truth kernels $\alpha$ and $\beta$. In green: the $N_{MC}$ reconstructed trajectories for each MLE. In orange: The average reconstructed kernels which are obtained using the coefficients $N_{MC}^{-1} \sum_{n=1}^{N_{MC}} (\hat a_m^{(v)})_n$ and  $N_{MC}^{-1} \sum_{n=1}^{N_{MC}} (\hat b_m^{(v)})_n$.}
    \label{fig:kernels_a_b_ws}
\end{figure}

We point out that despite the closeness between the estimated kernels $(\hat \alpha, \hat \beta)$ and true kernels $(\alpha,\beta)$, the estimated coefficients $(\hat a,\hat b)$ are very different from the ground truth $(a,b)$. For instance, for the $7$-th season we have the results reported in Table \ref{tab:7thseasonresult}.
\begin{table}[htb]
\centering
\begin{tabular}{lcccc}
\toprule
%
 Ground truth $b^7$ & -1.17 & -1.17 & 3.13 & 3.9 \\ [0.5ex] 
 \hline
 \text{Average } $\hat b ^7$ &  -0.81 & -4.6 & 12 & -2.85 \\ 
 \hline
 \text{Standard deviation } $\hat b ^7$&  3.1 & 35.7 & 114.5 & 105.4 \\[1ex] 
\bottomrule
\end{tabular}
\caption{Estimation results for the 7th season.}
\label{tab:7thseasonresult}
\end{table}
This means that even though the MLE is theoretically identifiable in the coefficients $a$ and $b$, it is weakly identifiable in practice. This is because two very different sets of coefficients $b$ and $b'$ can still produce two close kernels $\beta$ and $\beta'$. This weakness is not an issue as the MLE is strongly identifiable in $\alpha$ and $\beta$, which are the quantities that matter. 
\subsubsection{Estimation for the misspecified light-tailed model}
We now examine the effect of estimating a Type I network Poisson autoregression with kernels that are not exponential polynomials by maximising the Markov likelihood \eqref{eq:likelihood_markov} for a fixed characteristic time $\tau>0$ and order $q\in \N^*$. The considered kernels in this section are light-tailed, that is they vanish at least exponentially fast. For example, the momentum kernel is taken 

$$\alpha^{(v)}_k=(k-1)(k-2)(k-3)e^{-k}\frac{1+\cos(2\pi v/7)}{2},$$
and the network kernel 
$$\beta^{(v)}_k=\cos(5(k-1))\frac{e^{-0.1(k-1)^2}}{6}\sin(2\pi v/7).$$
Clearly, we no longer have ground truth coefficients $(a^{(v)}_q)$ and $(b^{(v)}_q)$, nor a ground truth characteristic time $\tau$ nor an order $q$. 

Assuming that we have the prior knowledge that the network ``forgets'' its state after one period that is $T_c=p=7$, we choose the characteristic time to be $\tau = 7\frac{3}{5}\simeq 4$, in accordance with Remark \ref{remark:tau}. The order of the Markov approximation is chosen $q=4$, that is, we maximise the likelihood as if $\alpha$ and $\beta$ were exponential polynomials of order $4$.

As with the previous subsection, we sample $N_{MC}=40$ simulations of Type I periodic network Poisson autoregression over $T=200$ periods. The results for the baseline pre-intensities $\mu^{(v)}$ are reported as boxplots in Figure \ref{fig:boxplot_mu_lt}.

\begin{figure}[h!]
    \centering
    \includegraphics[width=0.7\linewidth]{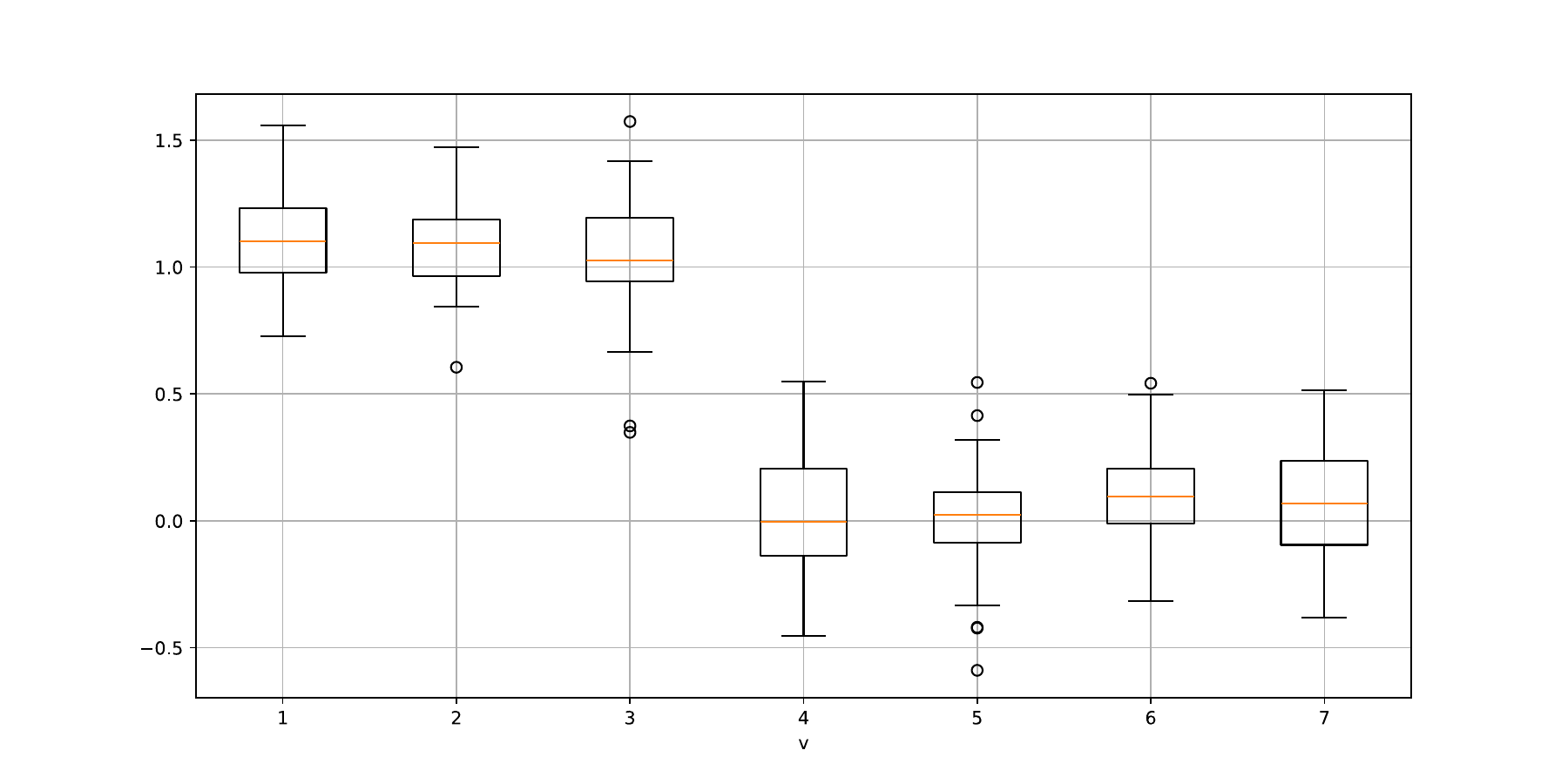}
    \caption{The estimated baseline pre-intensities. The average values are $(1.11, 1.10, 1.02, 0.04, 0.01, 0.10, 0.08)$ and the standard deviations are $(0.18, 0.18, 0.24, 0.22, 0.22,
       0.20, 0.22)$.}
    \label{fig:boxplot_mu_lt}
\end{figure}
Unlike the pre-intensities that are estimated for the well specified case reported in Figure \ref{fig:boxplot_mu}, it seems that for ill-specified case the $\hat \mu ^{(v)}$ are slightly overestimated. We believe that this overestimation is there to compensate an underestimation in the self-excitation caused by the momentum kernel, as it is shown in Figure \ref{fig:kernels_a_b_lt}.

\begin{figure}[h!]
    \centering
    \includegraphics[width=1\linewidth]{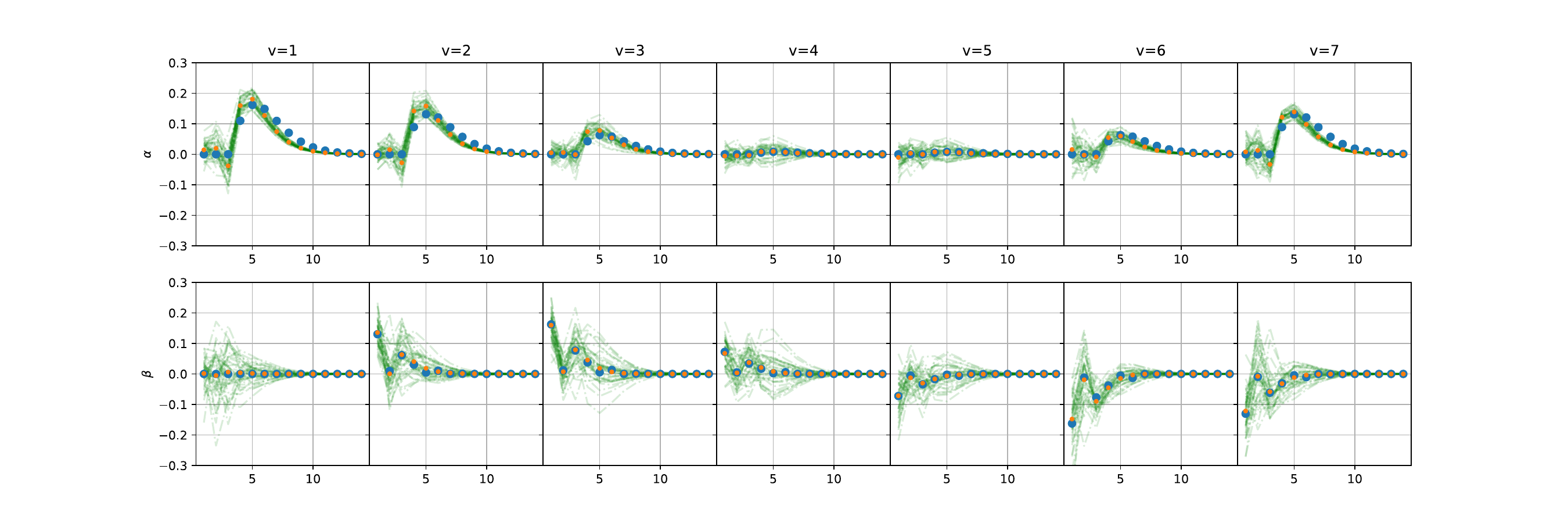}
    \caption{The characteristic time is $\tau=4$ and the order is $q=4$. In blue: the ground truth kernels $\alpha$ and $\beta$. In green: the $N_{MC}$ reconstructed trajectories for each MLE. In orange: The average reconstructed kernels which are obtained using the coefficients $N_{MC}^{-1} \sum_{n=1}^{N_{MC}} (\hat a_m^{(v)})_n$ and  $N_{MC}^{-1} \sum_{n=1}^{N_{MC}} (\hat b_m^{(v)})_n$.}
    \label{fig:kernels_a_b_lt}
\end{figure}
We notice that for the momentum kernel $\alpha$, the average estimated kernel decays slightly faster than $\alpha$. This underestimation is believed to stem from our choice of the characteristic time $\tau=4$ (corresponding to $T_c\simeq 7$): The true network seems to regress on its past values until $T_c\simeq 15$. However, despite this discrepancy in the characteristic time, the kernels $\alpha^{(v)}$ seem to be well captured by the approximation. 

As for the network kernels $\beta$, we notice that the variance is higher than for $\alpha$ or the kernels in the well-specified case. We believe that this is the case because $q=4$ decaying exponentials cannot easily capture the high non-monotonicity induced by the term $\cos(5(k-1))$. We now examine the effect of changing the characteristic time to $\tau=8$  which allows for a longer memory and of increasing the order to $q=6$. For the baseline pre-intensities, as shown in Figure \ref{fig:boxplot_mu_lt_bis}, we notice that they are no longer overestimated.

\begin{figure}[h!]
    \centering
    \includegraphics[width=0.7\linewidth]{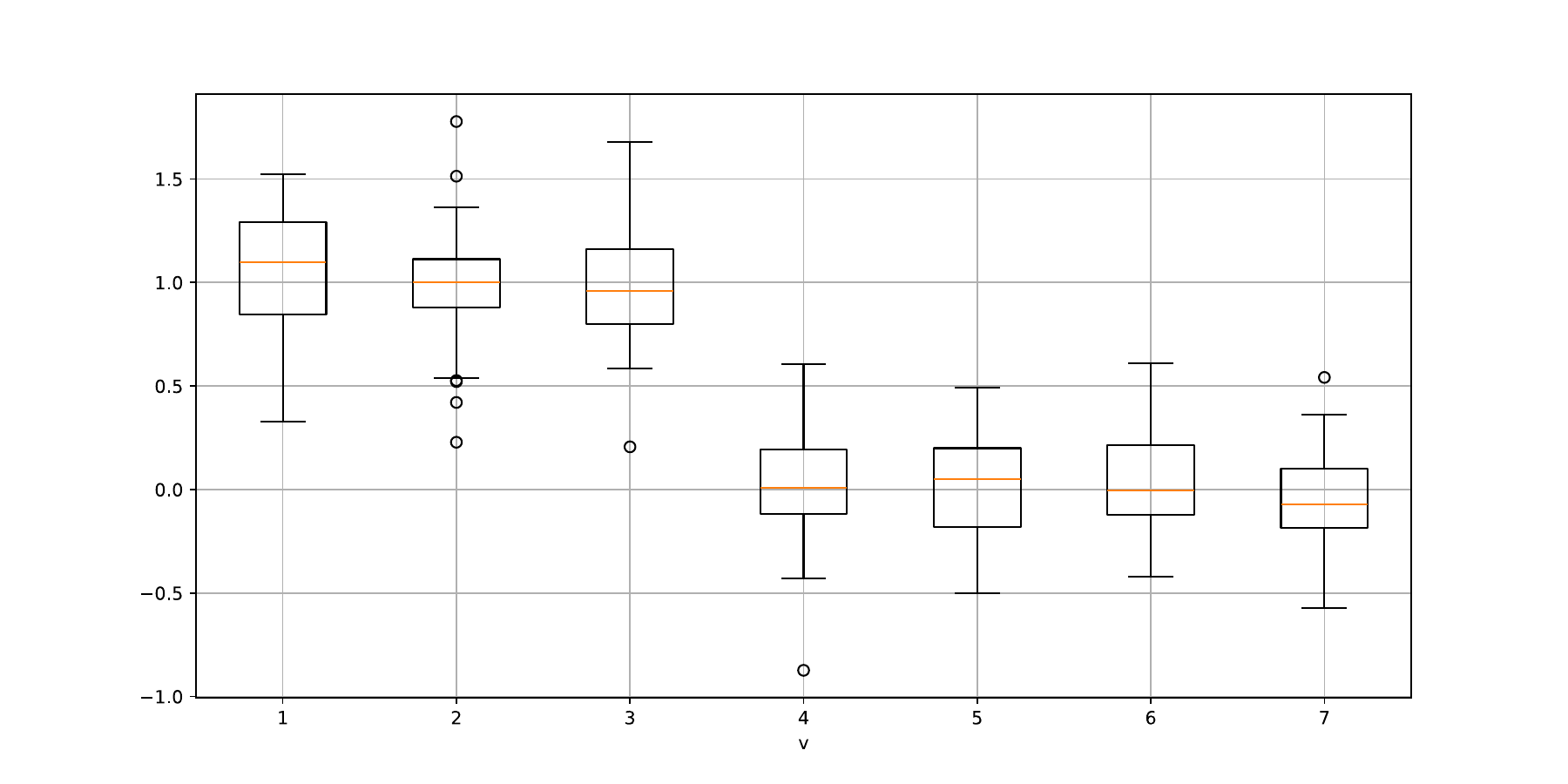}
    \caption{The MLE no longer overestimates the baseline intensities. The average values are $( 1.06,  0.98,  0.97,  0.01,  0.01,
        0.03, -0.05)$ and the standard deviations are $(0.29, 0.28, 0.26, 0.26, 0.24,
       0.24, 0.23)$.}
    \label{fig:boxplot_mu_lt_bis}
\end{figure}
The momentum kernels are also better captured with the more adapted choice $\tau=8$, as illustrated on Figure \ref{fig:kernels_a_b_lt_bis}. However, we notice that the estimation of the network kernels $\beta$ has a higher variance. This could be due to fact that $\tau=8$ yields an exponential that vanishes too slowly compared to $\beta$'s extinction time ($T_c\simeq 15$ \textit{vs.} $T_c\simeq 5$). The average kernel (lower panel, in orange) still captures the shape of the ground truth kernel quite well.

\begin{figure}[h!]
    \centering
    \includegraphics[width=1\linewidth]{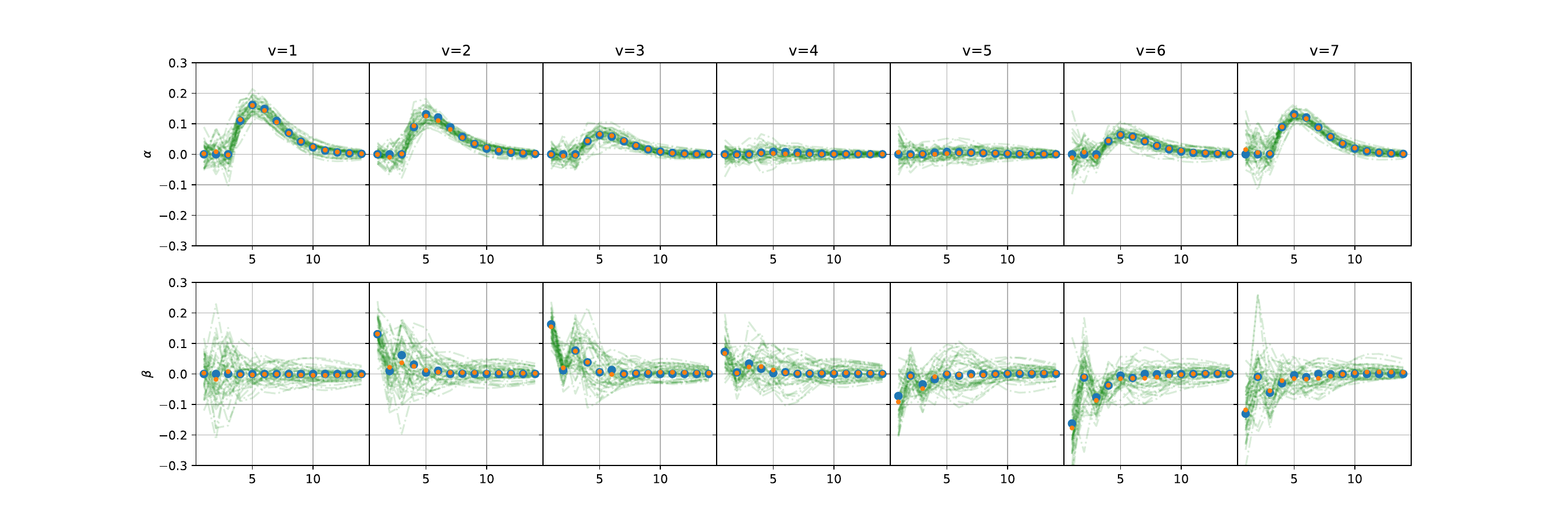}
    \caption{The characteristic time is now $\tau=8$ and the order has been increased to $q=6$. In blue: the ground truth kernels $\alpha$ and $\beta$. In green: the $N_{MC}$ reconstructed trajectories for each MLE. In orange: The average reconstructed kernels which are obtained using the coefficients $N_{MC}^{-1} \sum_{n=1}^{N_{MC}} (\hat a_m^{(v)})_n$ and  $N_{MC}^{-1} \sum_{n=1}^{N_{MC}} (\hat b_m^{(v)})_n$.}
    \label{fig:kernels_a_b_lt_bis}
\end{figure}
\subsubsection{Estimation for the misspecified heavy-tailed model}
Heavy-tailed kernels are any element of $\ell_1(\N^*)$ whose decay is slower than that of an exponential, \textit{e.g.} power law tails. Poisson autoregressions with heavy tails do not reach their periodically stationary regime as fast as autoregressions with light tails, as illustrated in Figure \ref{fig:approx}. This should yield \textit{a priori} a larger misspecification error with the Markov MLE. We now examine the performance numerically, with a momentum kernel taken as
$$\alpha^{(v)}_k=\frac{k^{1.6}}{6(1+0.2 k ^{3.6})}\frac{1+\cos(2\pi v/7)}{2},$$
and 
$$\beta^{(v)}_k=\frac{e^{-1.5 \sqrt{k-1}}k(\arctan(k-1.5)-0.5)}{5}\sin(2\pi v /7).$$

The Markov log-likelihood is maximised with a characteristic time $\tau=8$ and an order $q=6$. The estimation results for the baseline pre-intensities are reported in Figure \ref{fig:boxplot_mu_ht_bis}.

\begin{figure}[h!]
    \centering
    \includegraphics[width=0.7\linewidth]{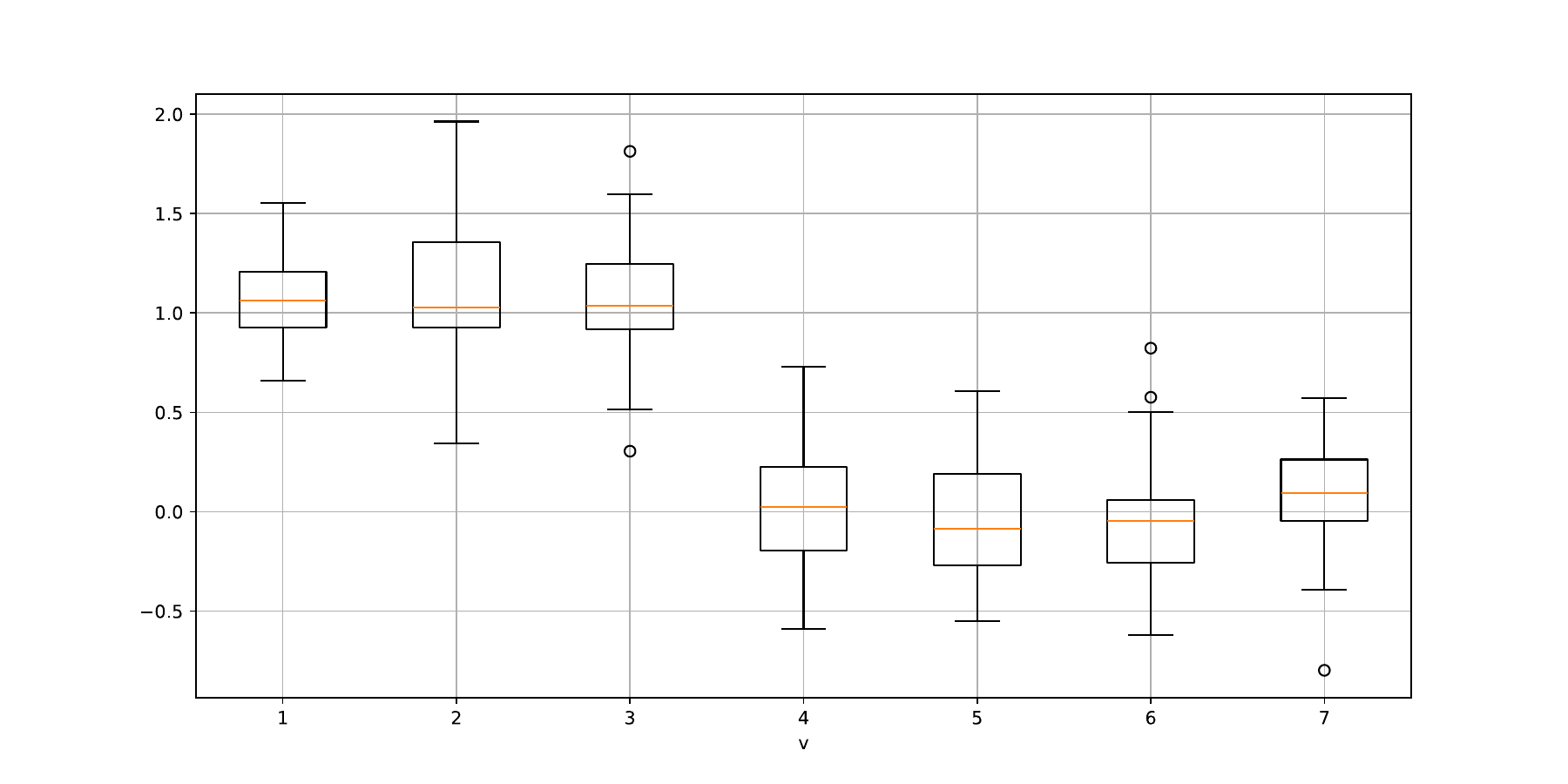}
    \caption{The estimated baseline pre-intensities. The average values are $(1.08,  1.11,  1.07,  0.02, -0.04, -0.06,  0.09)$ and the standard deviations are $(0.20, 0.30, 0.31, 0.32 , 0.29,
       0.30, 0.27).$}
    \label{fig:boxplot_mu_ht_bis}
\end{figure}

As for the momentum and network kernels, the results are presented in Figure \ref{fig:kernels_a_b_ht}.
\begin{figure}[h!]
    \centering
     \includegraphics[width=1\linewidth]{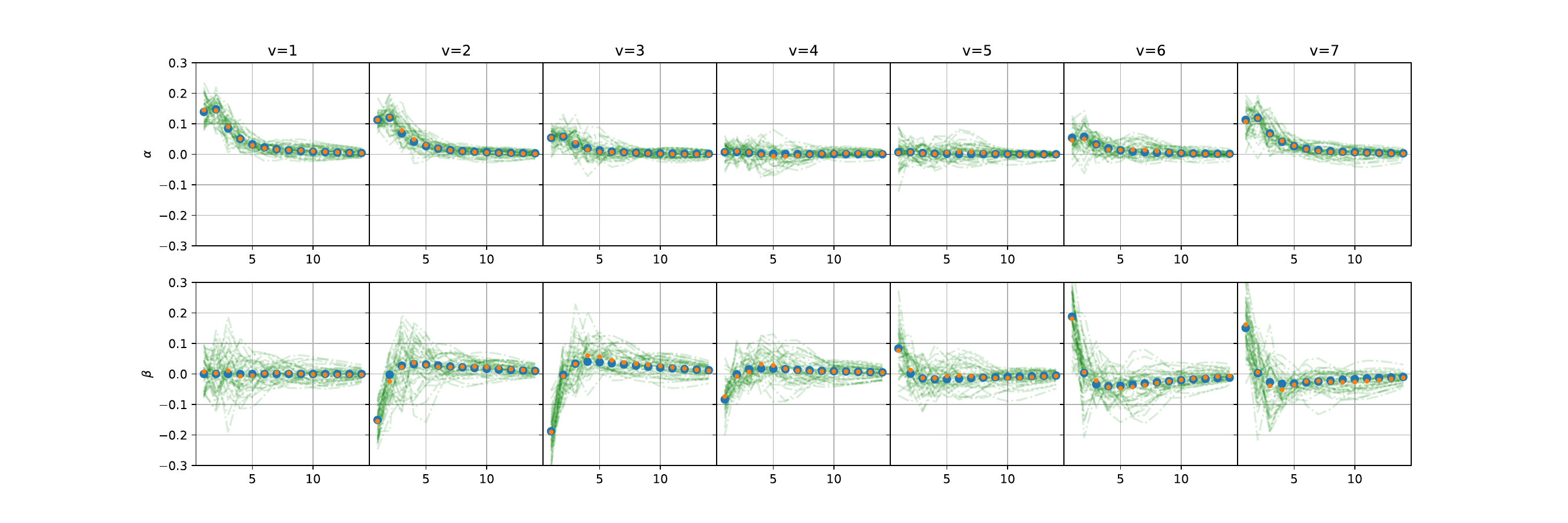}
    \caption{The characteristic time is $\tau=8$ and the order is $q=6$. In blue: the ground truth kernels $\alpha$ and $\beta$. In green: the $N_{MC}$ reconstructed trajectories for each MLE. In orange: The average reconstructed kernels which are obtained using the coefficients $N_{MC}^{-1} \sum_{n=1}^{N_{MC}} (\hat a_m^{(v)})_n$ and  $N_{MC}^{-1} \sum_{n=1}^{N_{MC}} (\hat b_m^{(v)})_n$.}
    \label{fig:kernels_a_b_ht}
\end{figure}
The Markov MLE seems to perform well on the heavy-tailed kernels as well, managing to capture both the baseline pre-intensities and the kernels. We notice however that the variance is higher than that of the well-specified or the misspecified light-tailed estimations. This is to be expected, as the heavy-tailed processes reach their periodically stationary distribution considerably slower than the aforementioned two models. Indeed, as it is shown in Figure  , increasing the time horizon from $T=200$ to $T=400$ leads to green curves that are more concentrated around the ground truth kernels
\begin{figure}[h!]
    \centering
    \includegraphics[width=1\linewidth]{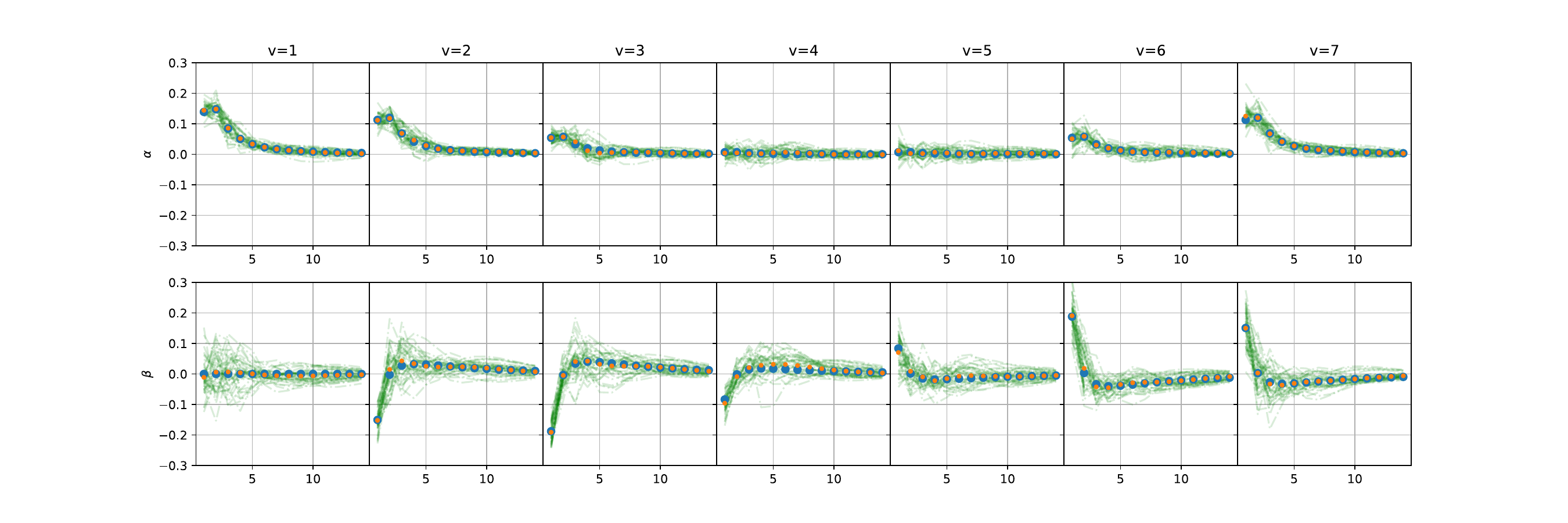}
    \caption{Increasing the time horizon from $T=200$ to $T=400$ leads to a visible reduction in the variance. In blue: the ground truth kernels $\alpha$ and $\beta$. In green: the $N_{MC}$ reconstructed trajectories for each MLE. In orange: The average reconstructed kernels which are obtained using the coefficients $N_{MC}^{-1} \sum_{n=1}^{N_{MC}} (\hat a_m^{(v)})_n$ and  $N_{MC}^{-1} \sum_{n=1}^{N_{MC}} (\hat b_m^{(v)})_n$.}
    \label{fig:kernels_a_b_ht_T400}
\end{figure}
\section{Rotavirus data analysis}
\label{sec:ral_data}
We now consider weekly cases of Rotavirus among children in Berlin between 2001 and 2015, a time interval of $T=732$ weeks in $d=12$ districts. The original data set covers the entire country of Germany ($412$ districts) and was obtained from \url{https://github.com/ostojanovic/BSTIM}. 
We will first fit both our seasonal Markov model of order $1$ and a PNAR($1$) model using likelihood maximisation in the first $11$ years (573 weeks, roughly $80\%$ of the data) and compare both values of the Bayes Information Criterion (BIC). Then we use both models to forecast the weekly number of cases of Rotavirus for each of the $12$ districts and compare how they perform compared to each other.
\subsection{Model comparison}
Throughout this section, we only consider linear Poisson autoregressions, that is, $\psi(x)=x$. Linearity ensures that quantities such as the expected value are exactly computable, but comes at the price of not allowing self/mutual inhibition. This is not a problem, as contagious viruses tend to trigger more cases rather than surpress them. 

The neigbourhood structure is straighforward; two nodes (\textit{i.e.} districts) have an edge between them if the corresponding districts share a border. This gives us the column-normalised weighted adjacency matrix $W$.

\begin{figure*}[h!]
    \centering
    \begin{subfigure}[t]{0.5\textwidth}
        \centering
        \includegraphics[height=56mm]{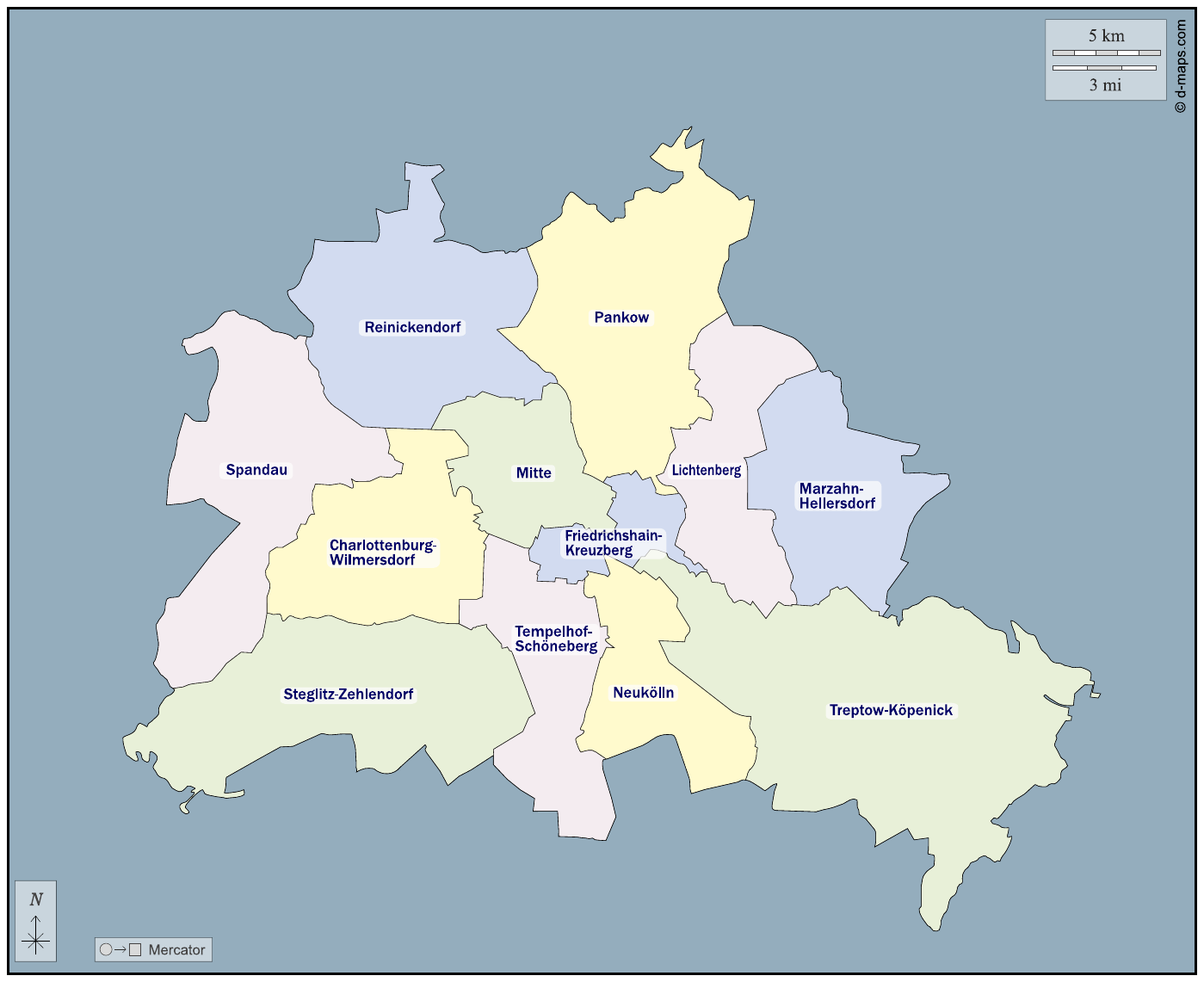}
    \end{subfigure}%
    ~ 
    \begin{subfigure}[t]{0.5\textwidth}
        \centering
        \includegraphics[height=55mm]{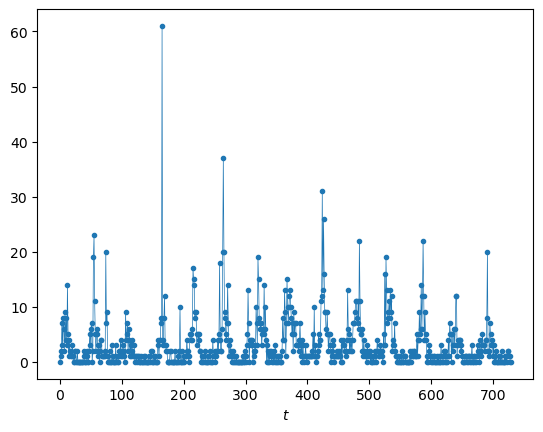}
    \end{subfigure}
    \caption{Left: A map of Berlin taken from \url{https://d-maps.com}. Right: Weekly number of Rotavirus cases in the Mitte district.}
\end{figure*}
Since diseases tend to be seasonal, the model we propose is 
\begin{equation}
\label{eq:model_1}
    \begin{cases}
        Y_t&=N_t\left(\mu_t+ (a^{(t)}I_{12}+b^{(t)}W)\xi_t\right),\\
        \xi_t&=e^{-\frac{3}{\tau}}(\xi_{t-1}+Y_{t-1}),
    \end{cases}
    \end{equation}
where $\mu_t$ is a vector of dimension $12$ corresponding to each district's baseline intensity whereas $a^{(t)}$ and $b^{(t)}$ are scalar sequences encoding the momentum and network effects. All of the sequences are $p-$periodic ($p$ will be determined later) and for the sake of parsimony will be parametrised 
\begin{equation*}
\begin{cases}
    \mu_t&=\mu + \mu' \sin \left(\frac{2\pi t}{p}\right) + \mu'' \cos \left(\frac{2\pi t}{p}\right),\\
    a^{(t)}&=a + a' \sin \left(\frac{2\pi t}{p}\right) + a'' \cos \left(\frac{2\pi t}{p}\right),\\
    b^{(t)}&=b + b' \sin \left(\frac{2\pi t}{p}\right) + b'' \cos \left(\frac{2\pi t}{p}\right).\\
    \end{cases}
\end{equation*}
Regarding the choice of the period $p$, a simple choice would be fixing $p=52 \in \N^*$. However, a year will always have slightly more than $52$ weeks, we therefore choose $p=52.18\simeq 365.25/7$. While the fact that $p$ is no longer integer means that the model is no longer within the framework of periodic Poisson autoregressions studied in the previous sections, we still study it empirically. We fit the model by maximising the likelihood in the variables $(\mu,\mu',\mu'',a,a',a'',b,b',b'')$ nevertheless. When it comes to the decay parameter $\tau$, we assume that the bulk of infection $(95\%)$ happens within 3 weeks. Hence, we perform our analysis with $\tau=3 $ weeks.

As a first comparison with the PNAR($1$) model introduced in \cite{ArFok}
\begin{equation}
    \label{eq:model_2}
    \begin{cases}
    Y_t &=N_t \left(\lambda_t\right),\\
    \lambda_t &=\mu +(aI_{12}+bW)Y_{t-1},
    \end{cases}
\end{equation}
we compute the BIC values after maximising the likelihood using the \texttt{SLSQP} method of optimisation under Python. The values computed over the first $t=573$ weeks are reported in Table \ref{tab:comparison}.
\begin{table}[htb]
\centering
\begin{tabular}{lc}
\toprule
\text{Our model} \eqref{eq:model_1} & -16861.78\\ [0.5ex] 
 \hline
 \text{PNAR($1$) }\eqref{eq:model_2}  & -16159.77  \\ 
[1ex] 
 \bottomrule
\end{tabular}
\caption{Comparison of the BIC values for our model \eqref{eq:model_1} and the \text{PNAR($1$) } model, see \eqref{eq:model_2}. }
\label{tab:comparison}
\end{table}
\subsection{Forecasting}
Throughout this section, we compare the forecasting performance of our model to the PNAR process. For both models, the predictor of the number of cases at time $t+i$ knows the history until time $t$. We point out that for both models, including in the non-linear case, the one-step prediction can be explicitly computed as 
\begin{align*}\E[Y_{t+1}|\mathcal F^Y_t]=\lambda_{t+1}
=\psi\left(\mu_{t+1}+\sum_{k=1}^{t}\phi^{(t)}_{t+1-k}Y_k\right),
\end{align*}
where $\left(\phi^{(t)}_k\right)_{k \in \N^*}$ does not depend on $t$ and only takes $q$ values for the PNAR($q$) models. 

For the linear case, the higher order conditional expected values can be computed recursively. Indeed, for a given $j>1$, we have, using the tower property of the conditional expectation 
\begin{align*}
    \E[Y_{t+j}|\mathcal F^Y_t]&=\E \left [\E [Y_{t+j}|\mathcal F^Y_{t+j-1}]|\mathcal F^Y_t\right ]\\
    &=\E\left[ \mu_{t+j}+\sum_{k=1}^{t+j}\phi^{(t+j)}_{t+j-k}Y_k \bigg |\mathcal F^Y_t\right]\\
    &=\mu_{t+j}+\sum_{k=1}^{t}\phi^{(t+j)}_{t+j-k}Y_k + \sum_{k=t+1}^{t+j}\phi^{(t+j)}_{t+j-k}\E [Y_k|\mathcal F^Y_t].
\end{align*}
The conditional value will play the role of our predictor in this section.\\
After fitting both our model \eqref{eq:model_1} and PNAR($1$) \eqref{eq:model_2} by maximising the likelihood over the first $t=573$ weeks ($\simeq 80\%$ of the data), we forecast the next $h>1$ steps ahead. Once that time is reached, we repeat the same procedure, until we reach the last time $T=732$. More concretely, the predictor writes
$$\hat Y_{t+j}=\E \left [Y_{t+j}|\mathcal F^Y_{t+\lfloor j/h \rfloor h} \right].$$
The performance is measured by evaluating the Root Mean Square Error (RMSE) for each district $i$
$$\text{RMSE}^{(i)}=\left(\frac{1}{T-t}\sum_{j=1}^{T-t}\left|\hat Y^{(i)}_{t+j}-Y^{(i)}_{t+j}\right|^2\right)^{1/2}.$$
The performance of the two models is thus compared by comparing their respective RMSE for each district.

To determine whether the difference is statistically significant, we deploy the Diebold-Mariano (DM) test \cite{DM}, which establishes whether the null hypothesis of equal forecasting performance can be confidently rejected.
The DM test is run using the Python code found here \url{https://github.com/johntwk/Diebold-Mariano-Test}. The obtained $p-$values are then adjusted using the Benjamini-Hochberg (BH) procedure. 
\subsubsection{Short term forecasting}
We seek to predict the weekly number of Rotavirus cases in each district of Berlin over a horizon of $h=4$ weeks. Over this relatively short period (roughly one month), the seasonality is not expected to be very pronounced. For both our model and the PNAR ($1$) model, the RMSE values for each district are reported in Table \ref{tab:shorttermforc}. We also report the value of the DM test and the BH adjusted $p-$values.
\begin{table}[htb]
\centering
\begin{tabular}{lcccc}
\toprule
\text{District} & \text{RMSE} \eqref{eq:model_1} & \text{RMSE} \eqref{eq:model_2} & \text {DM value} & \text{Adjusted $p-$value}\\ [0.5ex] 
 \hline
 \text{Mitte } & 2.71 &2.62 & 0.66 & 0.552  \\ 
 \text{Friedrichshain-Kreuzberg} & 1.70 & 1.91 & -0.91 & 0.482 \\
 \text{Pankow} & 7.31 & 7.48 & -2.14 & 0.133\\
 \text{Charlottenburg-Wilmersdorf} & 4.09 & 4.32 & -1.16 & 0.372 \\
 \text{Spandau} & 2.72 & 2.89 & -1.60 & 0.250\\
 \text{Steglitz-Zehlendorf} & 4.84 & 5.07 & -1.54 & 0.250 \\
 \text{Tempelhof-Schöneberg} & 2.66 & 2.60 & 0.70 & 0.552 \\
 \text{Neukölln} & 2.90 & 3.08 & -1.16 & 0.372 \\
 \text{Treptow-Köpenick} & 3.24 & 3.70 & -1.95 & 0.158 \\
 \text{Marzahn-Hellersdorf} & 2.34 & 3.44 & -2.93 & \textbf{0.046} \\
 \text{Lichtenberg} & 2.06 & 2.64 & -2.39 & 0.108 \\
 \text{Reinickendorf} & 1.69 & 1.74 & -0.46 & 0.645 \\ 
[1ex] 
\bottomrule
\end{tabular}
\caption{Results from predicting the weekly number of Rotavirus cases in each district of Berlin over a horizon of $h=4$ weeks. The first and second columns present the RMSE values for each district for our model and the PNAR ($1$) model, respectively.  The remaining columns present the values of the DM test and the BH adjusted $p-$values.}
\label{tab:shorttermforc}
\end{table}
The adjusted $p-$values below the confidence level of $5\%$ are given in bold.

We notice that the RMSE values are lower for our model in all districts except for Mitte and Tempelhof-Schöneberg. This advantage is not sufficient to reject the hypothesis of equal forecasting performance. 
\subsubsection{Long term forecasting}
We now forecast the weekly number of cases over the longer horizon of $h=13$ weeks (roughly three months). It is expected for the seasonality to be more marked over such a longer horizon. The numerical values are reported in the Table \ref{tab:longtermforc}.
\begin{table}[htb]
\centering
\begin{tabular}{lcccc}
\toprule
\text{District} & \text{RMSE} \eqref{eq:model_1} & \text{RMSE} \eqref{eq:model_2} & \text {DM value} & \text{Adjusted $p-$value}\\ [0.5ex] 
 \hline
 \text{Mitte} & 3.67 & 2.89 & 2.91 & $\underline{\mathbf{1.23 \cdot 10^{-2}}}$  \\ 
 \text{Friedrichshain-Kreuzberg} & 2.57 & 2.43 & 0.62 & 0.707 \\
 \text{Pankow} & 6.74 & 7.12 & -2.76 & $\mathbf{1.55 \cdot 10^{-2}}$\\
 \text{Charlottenburg-Wilmersdorf} & 4.09 & 4.70 & -1.86 & 0.096 \\
 \text{Spandau} & 2.99 & 3.07 & -0.31 & 0.763\\
 \text{Steglitz-Zehlendorf} & 5.63 & 6.66 & -3.28 & $\mathbf{0.51 \cdot 10^{-2}}$ \\
 \text{Tempelhof-Schöneberg} & 3.24 & 3.60 & -2.19 & $5.03 \cdot 10^{-2}$ \\
 \text{Neukölln} & 2.62 & 3.56 & -3.63 & $\mathbf{0.23 \cdot 10^{-2}}$ \\
 \text{Treptow-Köpenick} & 5.93 & 6.00 & -0.41 & 0.763 \\
 \text{Marzahn-Hellersdorf} & 2.55 & 3.07 & -2.59 & $\mathbf{2.05 \cdot 10^{-2}}$ \\
 \text{Lichtenberg} & 2.52 & 3.27 & -4.64 & $\mathbf{5.8 \cdot 10^{-5}}$ \\
 \text{Reinickendorf} & 2.43 & 2.49 & -0.30 & 0.763 \\ 
[1ex] 
\bottomrule
\end{tabular}
\caption{Results from predicting the weekly number of Rotavirus cases in each district of Berlin over a horizon of $h=13$ weeks. The first and second columns present the RMSE values for each district for our model and the PNAR ($1$) model, respectively.  The remaining columns present the values of the DM test and the BH adjusted $p-$values.}
\label{tab:longtermforc}
\end{table}
The values given in bold are the adjusted $p-$values below the confidence level of $5\%$. The underlined value corresponds to the only district (Mitte) in which the PNAR(1) model predicts the weekly number of cases significantly better than our model.   

We can then conclude that our model \eqref{eq:model_1} performs significantly better than the PNAR(1) process in 5 of the 12 districts and that it performs significantly worse in the Mitte district. For the remaining 6 districts, the null hypothesis of equal prediction performance cannot be rejected. That being said, in 5 of them, our model has a smaller RMSE than the PNAR(1) model. 

As a final illustration, we show the performance of the predictions based on our model and on a PNAR(1) dynamics, in Mitte and Lichtenberg.  

\begin{figure*}[h!]
    \centering
    \begin{subfigure}[t]{0.5\textwidth}
        \centering
        \includegraphics[height=60mm]{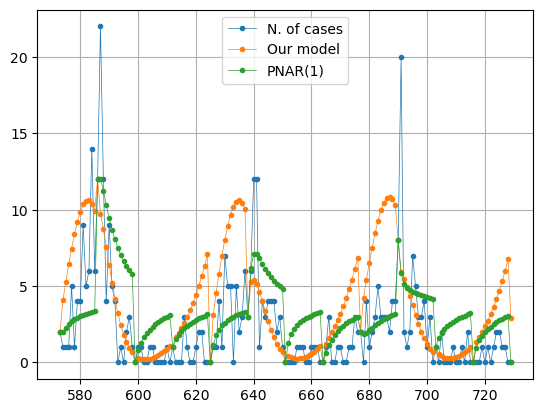}
        \caption{Mitte.}
    \end{subfigure}%
    ~ 
    \begin{subfigure}[t]{0.5\textwidth}
        \centering
        \includegraphics[height=60mm]{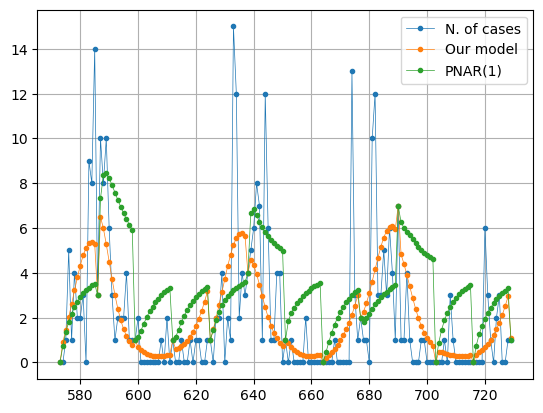}
        \caption{Lichtenberg.}
    \end{subfigure}
    \caption{Weekly predicted number of cases on a horizon of 13 weeks. For Mitte, seasonality seems to hinder accurate prediction because of some lag in the real cases compared to past seasons. For Lichtenberg, seasonality considerably enhances prediction's accuracy .}
\end{figure*}

\section{Conclusion}
In this article, we studied two types of periodic network autoregressions, for which we gave sufficient stability conditions as well as bounds on the speed of convergence to their periodically stationary regimes. Furthermore, we proved that the infinite memory process admits a Markov approximation that reduces the simulation and inference cost from quadratic to linear. The empirical study shows that such an approximation is robust when it comes to inference, even if the original dynamics have a heavy tail. 
We believe that there is still room for improvement for the different results presented in this paper, especially in two areas:
\begin{enumerate}
    \item By relaxing the stability Assumption \ref{ass:stability} in case the kernel has infinite memory to a ``contraction over a period'' condition rather than a ``contraction along every season'', or by proving a condition on the positive part of the kernel rather than its absolute value.
    \item By providing guarantees on the convergence of the MLE for the misspecified case.
\end{enumerate}

These results can naturally be extended to the ``risk'' network autoregression, which is more adapted to insurance problems. Indeed, if instead of simply counting events we can aggregate the ``losses'' from them 
$$R_t^{(i)}=\sum_{n=1}^{Y_t^{(i)}}Q^{(i,n)}_t$$
where $Q^{i,n}_t$ are an \textit{iid} family of random variables of common distribution $\kappa$ that play the role of an insurance claim. The stochastic intensity can also depend on the ``severity'' of the claims:
$$\lambda_t=\psi \left( \mu_t + \sum_{k=1}^{t-1}(\alpha^{(t)}_{t-k}I_d+\beta^{(t)}_{t-k}W)R_k\right).$$
The stability condition becomes $L\E[Q]\rho\left(\sum_{k\geq 1}\max_{v=1,\cdots,p}(|\alpha_k^{(v)} |I_d + |\beta^{(v)}_k| W)\right)<1 $ and the periodic stationarity results as well as those of the Markov approximation generalise naturally. The model can also be extended to include exogenous noise whether of discrete or continuous nature, simply by adding an extra term to the intensity's expression, for example
$$\lambda_t=\psi \left( \mu_t + \sum_{k=1}^{t-1}(\alpha^{(t)}_{t-k}I_d+\beta^{(t)}_{t-k}W)(R_k+\sigma X_k)\right),$$
where $(X_k)_{k\in \N^*}$ is an \textit{iid} family of standard Gaussian or Poisson variables. 

Finally, we notice that the field of continuous time periodic stochastic processes remains relatively explored (\textit{cf.} \cite{Veraart2024} for periodic trawl processes and \cite{HL} for periodically forced Markov processes). It might then be worthwhile to use the periodic Poisson autoregression presented in this paper as a way of constructing a periodic Hawkes process $(H_t)_{t\in \R}$ of intensity
$$\lambda_t=\psi\left(\mu_t+\int_{-\infty}^{t-}\phi(t,t-s)\mathrm d H_s\right)$$
for Type I periodicity and 
$$\lambda_t=\psi\left(\mu_t+\int_{-\infty}^{t-}\phi(s,t-s)\mathrm d H_s\right)$$
for Type II periodicity, $\phi$ here being a function that is $p$ periodic in its first argument.  

\subsection*{Acknowlegement}
The authors wish to thank Abdelhakim Aknouche, Olivier Wintenberger and Wei Wu for their useful suggestions. This work was supported by the EPSRC NeST Programme grant EP/X002195/1.

\subsection*{Authors contribution statement using the CRediT classification:} 
Mahmoud Khabou: Conceptualization, Data Curation, Formal Analysis,  Investigation, Methodology, Software, Writing – Original Draft Preparation;
Ed Cohen: Conceptualization, Funding Acquisition, Methodology, Supervision, Writing – Review \& Editing;
Almut Veraart: Conceptualization, Funding Acquisition, Methodology, Supervision, Writing – Review \& Editing.
All authors read and approved the final manuscript.

\appendix
\section{Proofs}
\subsection{Proof of Proposition \ref{prop:truncate}}
\label{sec:prop:truncate}
  The proof relies on showing that \eqref{eq:finite_reg} is a contraction on average over a period, which yields the existence and uniqueness of a stationary solution. Following the steps of \cite{DT}, we start by introducing the random map from $\R^{mp}$ to itself 
    $$F_t: (x_1,x_2,\cdots,x_{mp}) \mapsto \left(f_t(x_1,x_2,\cdots, x_{mp}, 0,\cdots; \zeta_t),x_1,\cdots,x_{mp-1}\right)$$
    for any $v\in \{1,\cdots,p\}$. Clearly, $F_t$ is periodic in distribution and for any $v=1,\cdots,p$ 
    \begin{align*}
        \E \left[ \left | F_v(x_1,\cdots,x_{mp})-F_v(x'_1,\cdots,x'_{mp}) \right|_{ } \right] \preceq \Gamma_v \left| (x_1,\cdots,x_{mp})- (x'_1,\cdots,x'_{mp}) \right|_{ },
    \end{align*}
    where the matrices $\Gamma_1,\cdots, \Gamma_p$ are given in Assumption \ref{ass:contraction_periodic}.
    For $n \in \Z$, let 
    $$U_{n}=\left(X_{np},X_{(n-1)p+p-1},\cdots,X_{(n-m)p+1}\right)$$
    be the concatenation of $m$ consecutive seasons. This process satisfies $$U_{n+1}=G\left(U_n\right),$$
    where $G=F_p \circ F_{p-1}\circ \cdots \circ F_1$. Given two histories $\boldsymbol u = (x_1,x_2,\cdots,x_{mp})$ and $\boldsymbol{u}'=(x'_1,x'_2,\cdots, x'_{mp})$, we have that 
    \begin{align*}
        \E \left [\left| G(\boldsymbol u)-G(\boldsymbol u ')\right|_{ }\right] &= \E \left [ \E_{\zeta_1,\cdots,\zeta_p} \left [ \left | F_p \left( F_{p-1}\circ \cdots \circ F_1 (\boldsymbol u)\right)-F_p \left( F_{p-1}\circ \cdots \circ F_1 (\boldsymbol u')\right)\right|_{ }\right]\right]\\
        &\preceq \Gamma_p \E \left[ F_{p-1}\circ \cdots \circ F_1 (\boldsymbol u)- F_{p-1}\circ \cdots \circ F_1 (\boldsymbol u') \right],
    \end{align*}
     Repeating the same conditioning and bounding from above yields 
     \begin{align*}
          \E \left [\left| G(\boldsymbol u)-G(\boldsymbol u ')\right|_{ }\right] & \preceq (\Gamma_p\cdots \Gamma_1) |\boldsymbol{u}-\boldsymbol{u}'|_{ } 
     \end{align*}
     where the product of the matrix $\Gamma_p\cdots \Gamma_1$ has a spectral radius strictly smaller than $1$. Using Gelfand's formula there exists $\bar \rho \in (0,1)$ and $q \in \mathbb N$ such that 
     \begin{equation}
     \label{ineq:contraction}
     \E \left \|G^n(\boldsymbol{u})-G^n(\boldsymbol{u}') \right \|_1 \leq \bar\rho^n \|\boldsymbol{u}-\boldsymbol{u}'\|_1
     \end{equation}
     whenever $n\geq q$. Using Theorem 2 in \cite{WS} we conclude regarding the existence and weak dependence of a stationary process $(U_n)_{n\in \Z}$.
     Indeed, using the coupling argument of \cite{DP}, the weak dependence coefficient $\tau_n$ is bounded from above by $C\bar \rho ^n$ for some positive constant $C$. The uniqueness and exponentially fast convergence towards the stationary solution are deduced from Theorem 2.6.1 in \cite{Straumann}.
     \subsection{Proof of Theorem \ref{thm:infinite}}
     \label{sec:thm:infinite}
         First, we fix an \textit {ipd} (independent and periodically distributed) sequence of innovations $(\zeta_t)_{t\in \Z}$, that is, an independent sequence such that $\zeta_{t+p}$ has the same distribution as $\zeta_t$, for all $t\in \Z$. Let $q\geq m$ be two positive integers and let $\tilde X^{(m)}$ and $\tilde X^{(q)}$ be the periodically stationary solutions of
     $$\tilde X^{(m)}_t=f_t(\tilde X^{(m)}_{t-1},\tilde X^{(m)}_{t-2},\dots,\tilde X^{(m)}_{t-mp},0,\cdots; \zeta_t)$$
     and 
     $$\tilde X^{(q)}_t=f_t(\tilde X^{(q)}_{t-1},\tilde X^{(q)}_{t-2},\dots,\tilde X^{(q)}_{t-qp},0,\cdots; \zeta_t),$$
     respectively. The existence of both of these processes is given by Proposition \ref{prop:truncate}. Assumption \ref{ass:contraction} yields that for any $t \in \Z$ 
     \begin{align}
         \E \left[ |\tilde X^{(m)}_t-\tilde X^{(q)}_t|_{ }\right] &= \E \left [ |f_t(\tilde X^{(m)}_{t-1},\tilde X^{(m)}_{t-2},\dots,\tilde X^{(m)}_{t-mp},0,\cdots; \zeta_t)-f_t(\tilde X^{(q)}_{t-1},\tilde X^{(q)}_{t-2},\dots,\tilde X^{(q)}_{t-qp},0,\cdots; \zeta_t)|_{ }\right]\nonumber \\
         &= \E \left [\E_{t-1} |f_t(\tilde X^{(m)}_{t-1},\tilde X^{(m)}_{t-2},\dots,\tilde X^{(m)}_{t-mp},0,\cdots; \zeta_t)-f_t(\tilde X^{(q)}_{t-1},\tilde X^{(q)}_{t-2},\dots,\tilde X^{(q)}_{t-qp},0,\cdots; \zeta_t)|_{ }\right]\nonumber \\
         &\preceq \sum_{k=1}^{mp}A_k \E |\tilde X^{(m)}_{t-k}-\tilde X^{(p)}_{t-k}|_{ } + \sum_{k=mp+1}^{qp} A_k \E |\tilde X^{(q)}_{t-k}|_{ }. \label{ineq:cauchy}
     \end{align}
     Since $\tilde X^{(q)}$ is periodic stationary, there exists $0\preceq C_q $ such that $\E |\tilde X^{(q)}_t|_{ }\preceq C_q$ for any $t\in \Z$ and 
     \begin{align*}
         \E |\tilde X^{(q)}_v|_{ } &\preceq \E |f_t(\tilde X^{(q)}_{v-1},\cdots,\tilde X^{(q)}_{v-qp},0,\cdots; \zeta_0)-f_t(0,\cdots; \zeta_0)|_{ } +\E |f_v(0,\cdots; \zeta_v)|_{ }\\
         &\preceq \sum_{k=1}^{qp} A_k \E|\tilde X^{(q)}_{v-k}|_{ } +\max_{i=1,\cdots p} \E |f_i(0,\cdots; \zeta_i)|_{ }\\
         &\preceq \sum_{k=1}^{+\infty} A_k C_q +|b|_{ },
     \end{align*}
     where $|b|_{ }= \max_{i=1,\cdots p} \E |f_i(0,\cdots; \zeta_i)|_{ }$. Thus, the upper bound $C_q$ satisfies 
     $$(I-S_A)C_q \preceq  |b|_{ },$$
     with $S_A=\sum_{k=1}^{+\infty} A_k$, a non-negative matrix with spectral radius strictly less than $1$. The inverse of $(I-S_A)$ is $\sum_{i=0}^{+\infty} S_A^{(i)}$ which is non-negative as well, hence multiplying by it does not change the sign of the inequality and we have that for any $t \in \Z$ and $q \in \N$
     \begin{equation}
     \label{ineq:terme-cst}
         \E |X^{(q)}_t|\preceq C_q \preceq (I-S_A)^{-1}|b|_{ }.
     \end{equation}
     Combining Inequalities \eqref{ineq:cauchy} and \eqref{ineq:terme-cst} yields 
     $$ \E \left[ |\tilde X^{(m)}_t-\tilde X^{(q)}_t|_{ }\right] \preceq \sum_{k=1}^{+\infty}A_k \E |\tilde X^{(m)}_{t-k}-\tilde X^{(p)}_{t-k}|_{ } + \sum_{k=mp+1}^{qp} A_k (I-S_A)^{-1}|b|_{ },$$
     which combined with Lemma \ref{lmm:convolution} yields
     $$\E |\tilde X^{(m)}_t-\tilde X^{(q)}_t|_{ } \preceq (I-S_A) ^{-1}\left(\sum_{k=mp+1}^{qp} A_k \right)(I-S_A)^{-1}|b|_{ }.$$
     The sequence $(\tilde X^{(m)})_{m\in \N}$ is thus a Cauchy sequence in the Banach space $L_1$, which means that it admits a unique limit $\tilde X$. Using the same arguments from Section 5.3 in \cite{DW}, we conclude regarding the periodic stationarity, measurability with respect to the filtration $\F_t = \sigma \left (\zeta_k, k\leq t\right)$ and the fact that $X$ solves \eqref{eq:inf_reg}.
\subsection{Proof of Proposition \ref{prop:speed}}
\label{sec:prop:speed}
  Using Assumption \ref{ass:contraction} we have that for any 
    \begin{align*}
        \E \left[|\tilde X_t - X_t|_{ } \right] & \preceq \sum_{k=1}^{+\infty} A_k \E [|\tilde X_{t-k}-X_{t-k}|_{ }]\\
        & = \sum_{k=1}^{t-1} A_k \E [|\tilde X_{t-k}-X_{t-k}|_{ }]+ \sum_{k=t}^{+\infty} A_k \E [|\tilde X_{t-k}-X_{t-k}|_{ }]\\
        & \preceq \sum_{k=1}^{t-1} A_k \E [|\tilde X_{t-k}-X_{t-k}|_{ }]+ \sum_{k=t}^{+\infty} A_k C,\\
    \end{align*}
    where $C$ is such that $x_{-k}+ \E[X_0] \preceq C$ for any $k \in \N$. Following the same techniques used in the proof of Lemma \ref{lmm:convolution} we have that 

    $$\E \left [|X_t-\tilde X_t|_{ } \right]\preceq \left(\sum_{k=1}^{t-1} B_kU_{t-k}\right)C,$$
    and since $\lim_{t \to +\infty} U_t = 0$ and $B\in \ell_1(\N^*)$, we have that $\lim_{t\to +\infty}\left(\sum_{k=1}^{t} B_kU_{t-k}\right)C=0$. \\
Assume now that for some $\beta >0$ and some nonnegative matrix $C$ such that $A_k \preceq C e^{-\beta k}$, then we have that 
\begin{equation}
\label{ineq:U_exp}
     U_t  =  \sum_{k=t}^{+\infty}A_k C \preceq C e^{-\beta t}. 
\end{equation}
Fix $\delta \in (0,\beta)$ and for any $n \in \N^*$, let $M_n:=\sum_{k=1}^{+\infty} e^{\delta k} A^{*n}_k$ which is an element of $\mathcal M_d([0,+\infty])$. We now prove that for $\delta$ small enough, $M_n$ is finite. For a given $n\in \mathbb N^*$ we have that
\begin{align*}
    M_{n+1}&=\sum_{k=1}^{+\infty} e^{\delta k} A^{*(n+1)}_k\\
    &=\sum_{k=1}^{+\infty} e^{\delta k}  \sum_{j=1}^{+\infty}A^{*n}_j A_{k-j} \boldsymbol{1}_{j\leq k-1}\\
    &=\sum_{j=1}^{+\infty} A^{*n}_j  \sum_{k=j+1}^{+\infty} e^{\delta k} A_{k-j}, 
\end{align*}
and, by a change of counter, we obtain that 
$$M_{n+1} = \sum_{j=1}^{+\infty} e^{\delta j}A^{*n}_j \sum_{k=1}^{+\infty} e^{\delta k}A_k  =M_n  \sum_{k=1}^{+\infty} e^{\delta k} A_{k}.$$
Since the function $g:\delta \to \rho \left( \sum_{k=1}^{+\infty} e^{\delta k} A_{k}\right)$ is continuous near zero and $g(0)<1$, there exists $\delta \in (0,\beta)$ such that  $\rho \left( \sum_{k=1}^{+\infty} e^{\delta k} A_{k}\right) <1$. Hence, $(M_n)_{n\in \N^*}$ decreases exponentially and therefore $$\sum_{k=1}^{+\infty} e^{\delta k} B_k = \sum_{n\geq 1} M_n < + \infty.$$
In this case, taking \eqref{ineq:U_exp} into consideration, we have that 
\begin{align*}
    \sum_{k=1}^{(t)} B_kU_{t-k} & = e^{-\delta t}\sum_{k=1}^{t-1} e^{k\delta}B_ke^{\delta(t-k)}U_{t-k}\\
    &\preceq C e^{-\delta t},
\end{align*}
which yields the exponential decay of the distance if the family $(A_k)_{k\in \N^*}$ vanishes exponentially fast. \\
Assume now that for some nonnegative matrix $C$ and $\beta>0$ we have that $A_k \preceq C k^{-2(1+\beta)}$, in this case the remainder sequence $U$ verifies 
\begin{align*}
    U_t &=\sum_{k=t}^{+\infty} k^{-(1+\beta)} k ^{1+\beta} A_kC\\
    & \preceq  t^{-(1+\beta)}\sum_{k=t}^{+\infty} k ^{1+\beta} A_kC\\
    & \preceq t^{-(1+\beta)}\sum_{k=t}^{+\infty} k ^{-(1+\beta)}C\\
    & \preceq  t^{-(1+2\beta)} C,
\end{align*}
which means that $U$ is in $\ell_1(\N)$. For a given $n\in\N^*$, define $M_n=\sum_{k=1}^{+\infty} k A^{*n}_k \in \mathcal M_d([0,+\infty])$. For a given $n \in \N^*$ we have that 
\begin{align*}
    M_{n+1} &= \sum_{k=1}^{+\infty}k A^{*(n+1)}_k\\
    &=\sum_{k=1}^{+\infty}k  \sum_{j=1}^{+\infty}A^{*n}_j A_{k-j} \boldsymbol{1}_{j\leq k-1}\\
    &=\sum_{j=1}^{+\infty}A^{*n}_j \sum_{k=1}^{+\infty} (k+j)A_k\\
    &=\left(\sum_{j=1}^{+\infty}A_j\right)^n M_1 + M_n \left(\sum_{k=1}^{+\infty}A_k\right),
\end{align*}
and since $\rho\left(\sum_{k=1}^{+\infty}A_k\right)<1$, we conclude that $(M_n)_{n\in \N^*} \in \ell_1(\N^*)$. therefore, $\sum_{k=1}^{+\infty}k B_k = \sum_{n\geq 1} M_k <+\infty$. We now proceed to bounding the convolution 
\begin{align*}
    \sum_{k=1}^{t-1} B_kU_{t-k} &= \sum_{k=1}^{\lfloor (t-1)/2 \rfloor} B_kU_{t-k} + \sum_{k=\lfloor (t-1)/2 \rfloor+1}^{t-1} B_kU_{t-k}\\
    &\preceq \sum_{k=1}^{\lfloor (t-1)/2 \rfloor} B_kC (t-k)^{-(1+2\beta)} +  \sum_{k=\lfloor (t-1)/2 \rfloor+1}^{t-1} Ck^{-1}U_{t-k}\\
    &\preceq C \left(\frac{2^{1+2\beta}}{t^{1+2\beta}} \|B\|_1+ \frac{2}{t} \|U\|_1 \right),
\end{align*}
which yields the desired result.

\subsection{Proof of Corollary \ref{cor:almost_sure}}
\label{sec:almost_sure}
    Without loss of generality we can assume that $X$ and $\tilde X$ are univariate, with the result for the multivariate case easily inferred. Let $\varepsilon >0$ and let $\delta \in (0,\beta)$ be as in Proposition \ref{prop:stationary}. Using Markov's inequality we have that 
    \begin{align*}
        \mathbb P \left(e^{\frac{\delta}{2}t}|\tilde X_t - X_t|\geq \varepsilon\right) &\leq \frac{e^{\frac{\delta}{2}t} \E[|\tilde X_t - X_t|]}{\varepsilon }\\
        &\leq \frac{e^{-\frac{\delta}{2}t}}{\varepsilon},
    \end{align*}
    which by Borel-Cantelli's Lemma yields that $e^{\frac{\delta}{2}t}|\tilde X_t - X_t|\to 0$ almost surely. The result follows immediately.

\subsection{Proof of Proposition \ref{prop:stationary}}
\label{sec:prop:stationary}
    Given the \textit{iid} (and thus \textit{ipd}) sequence of Poisson processes $(N_t)_{t\in \Z}$ and an infinite sequence of integer vectors $(x_1,x_2,\cdots)$, we define the periodic function 
    $$f_t(x_1,x_2,\cdots;N_t)=N_t \left(\psi\left(\mu _t + \sum_{k=1}^{+\infty}\phi^{(t)}_k x_k\right)\right).$$
    We clearly have that $\tilde Y$ as defined by Equation \eqref{def:TypeI_s} satisfies the recursion $$\tilde Y_{t}=f_{t}(\tilde Y_{t-1}, \tilde Y_{t-2},\cdots;N_t),$$
    and for $Y$ defined by Equation \eqref{def:TypeI} 
    $$ Y_{t}=f_{t}( Y_{t-1}, Y_{t-2},\cdots,Y_0,0,\cdots ;N_t).$$
    Clearly, we have that 
    \begin{align*}
        \E [|f_t (0,0,\cdots;N_t)|]&= \E [N_t \left(\psi (\mu_t)\right)]\\
        &=\psi(\mu_t)\\
        &\preceq \max_{v=1,\cdots,p} \psi(\mu_v),
    \end{align*}
    which is finite. Similarly, we have that for any $v \in \{1,\cdots,p\}$ 
    \begin{align*}
        \E[|f_v(x_1,x_2,\cdots;N_v)-&f_v(x'_1,x'_2,\cdots;N_v)|]\\
        &=\E \left[ \left| N_v \left( \psi \left(\mu_t + \sum_{k=1}^{+\infty}\phi^{(v)}_k x_k\right)\right)-N_v \left( \psi \left(\mu_t + \sum_{k=1}^{+\infty}\phi^{(v)}_k x'_k\right)\right)\right|\right]\\
        &=\E \left[ N_v \left ( \left | \psi\left(\mu_t+\sum_{k=1}^{+\infty}\phi^{(v)}_k x_k\right)-\psi \left(\mu_t+\sum_{k=1}^{+\infty}\phi^{(v)}_k x'_k \right)\right|\right)\right]\\
        &=\left | \psi\left(\mu_t+\sum_{k=1}^{+\infty}\phi^{(v)}_k x_k\right)-\psi \left(\mu_t+\sum_{k=1}^{+\infty}\phi^{(v)}_k x'_k \right)\right|,
    \end{align*}
    and using the fact that $\psi$ is $L-$Lipschitz we have 
    \begin{align*}
         \E[|f_v(x_1,x_2,\cdots;N_v)-f_v(x'_1,x'_2,\cdots;N_v)|]&\preceq L \left |\sum_{k=1}^{+\infty}\phi^{(v)}_k (x_k -x'_k) \right|\\
         &\preceq  \sum_{k=1}^{+\infty}L|\phi^{(v)}_k| \left |x_k -x'_k\right|\\
         &\preceq  \sum_{k=1}^{+\infty}A_k \left |x_k -x'_k\right|.\\
    \end{align*}
    Assumption \ref{ass:contraction} is thus in force and the results are obtained using Theorem \ref{thm:infinite} and Proposition \ref{prop:truncate}.\\
    For the moments, given a real number $r\geq 1 $ we introduce the vector adapted norm $\E^{1/r}[Y^r]:=(\E^{1/r}[Y_1^r],\cdots,\E^{1/r}[Y_d^r])$. Using the definition of $Y$, and applying the expectation and the powers component-wise we have that 
    \begin{align*}
        \E^{1/r}[Y^r_t]&=\E^{1/r} \left[N_t \left(\lambda_t\right)^r \right]\\
        &=\E^{1/r} \left[\E\left[N_t \left( \lambda_t\right)^r\Bigg | \mathcal F_{t-1}\right] \right].
    \end{align*}
    Let $\delta > 0$ be such that $\rho \left( (1+\delta)^{1/r}\sum_{k=1}^{+\infty} A_k\right)<1$. Using the proof of Lemma 2 in \cite{DT} we have that, component-wise 
    \begin{align*}
        \E^{1/r}[Y_t^r] &\preceq \E^{1/r} \left[(1+\delta)(\lambda_t)^r+C_{r,\delta} \mathds 1\right],
    \end{align*}
    where $\mathds 1$ is the vector with components $1$ and $C_{r,\delta}$ is a positive constant that does not depend on $\lambda_t$ and that can change from one line to the next. Using the fact that $(1+x)^{1/r}\leq 1+x^{1/r}$ for $x\geq 0$ we have that 
    \begin{align*}
          \E^{1/r}[Y_t^r] &\preceq (1+\delta)^{1/r} \E^{1/r} \left[(\lambda_t)^r +C_{r,\delta} \mathds 1\right]\\
          &= (1+\delta)^{1/r} C_{r,\delta}^{1/r}\E^{1/r} \left[\left(\frac{\lambda_t}{C_{r,\delta}^{1/r}}\right)^r +\mathds 1\right]\\
          &\preceq (1+\delta)^{1/r} C_{r,\delta}^{1/r} \left(\E^{1/r} \left[\left(\frac{\lambda_t}{C_{r,\delta}^{1/r}}\right)^r \right]+\mathds 1\right)\\
          &=(1+\delta)^{1/r} \E^{1/r} [(\lambda_t)^r] + C_{r,\delta} \mathds 1.
    \end{align*}
    Using Minkowski's inequality component-wise we obtain 
    \begin{align*}
        \E^{1/r}[Y_t^r] &\preceq (1+\delta)^{1/r} \E^{1/r}\left[\psi\left(\mu_t + \sum_{k=1}^{t-1} \phi^{(t)}_{t-k} Y_k\right)^r\right] + C_{r,\delta} \mathds 1 \\
        &\preceq (1+\delta)^{1/r}L \E^{1/r}\left[\left|\mu_t + \sum_{k=1}^{t-1} \phi^{(t)}_{t-k} Y_k\right |^r\right] + C_{r,\delta} \mathds 1 \\
        &\preceq (1+\delta)^{1/r}L\left(\mu_t + \sum_{k=1}^{t-1} |\phi^{(t)}_{t-k}|  \E^{1/r}\left[Y_k^r\right]\right) + C_{r,\delta} \mathds 1 \\
        &\preceq (1+\delta)^{1/r}L\left(\sum_{k=1}^{t-1} A_{t-k}  \E^{1/r}\left[Y_k^r\right]\right) + C_{r,\delta} \mathds 1.
    \end{align*}
    We thus conclude that the $r-$th moments are finite using Lemma \ref{lmm:convolution}.\\
    The same can be said about Type II periodicity (Equations \eqref{def:TypeII} and \eqref{def:TypeII_s}) using the periodic function $$f_t(x_1,x_2,\cdots;N_t)=N_t \left(\psi\left(\mu _t + \sum_{k=1}^{+\infty}\phi^{(t-k)}_k x_k\right)\right).$$ 

    \subsection{Proof of Proposition \ref{prop:continuity}}
    \label{sec:prop:continuity}
We start by proving the continuity result on Type I periodicity. Given $t \in \N^*$ we have
\begin{align*}
    \E \left[|Y_t-\bar Y_t| \right]&=\E \left[\left |N_t (\lambda_t) - N_t (\bar \lambda_t)\right | \right]\\
    &=\E \left [ \left | N_t \left (\psi \left(\mu_t + \sum_{k=1}^{t-1}\phi^{(t)}_{t-k}Y_k\right) \right) - N_t \left (\psi \left( \mu_t + \sum_{k=1}^{t-1}\bar\phi^{(t)}_{t-k} \bar Y_k\right) \right)\right |\right].
\end{align*}
By conditioning on $\mathcal F_{t-1}$ and using the fact that $\psi$ is $L-$Lipschitz 
\begin{align*}
     \E \left[|Y_t-\bar Y_t| \right]\preceq & L \E \left[\left |\sum_{k=1}^{t-1}\phi^{(t)}_{t-k}  Y_k-\bar\phi^{(t)}_{t-k} \bar Y_k \right| \right] \\
      \preceq & L \sum_{k=1}^{t-1}\E \left | \phi^{(t)}_{t-k}  Y_k-\bar\phi^{(t)}_{t-k} \bar Y_k\right|\\
     \preceq & L \sum_{k=1}^{t-1} |\phi^{(t)}_{t-k}|\E |Y_k-\bar Y_k| + L \sum_{k=1}^{t-1} |\phi^{(t)}_{t-k}-\bar \phi^{(t)}_{t-k}|\E|\bar Y_k| \\
     \preceq & L \sum_{k=1}^{t-1}A_{t-k}\E |Y_k-\bar Y_k| + L \sum_{k=1}^{t-1} |\phi^{(t)}_k-\bar \phi^{(t)}_k|C, \\
\end{align*}
where $C$ is an upper bound on $\E \bar Y_t$ (\textit{cf.} Proposition \ref{prop:stationary}). Bounding $ \sum_{k=1}^{t-1} |\phi^{(t)}_k-\bar \phi^{(t)}_k|$ from above by $\max_{v=1,\cdots,p}\|\phi^{(v)}-\bar \phi^{(v)}\|$ and using Lemma \ref{lmm:convolution} yield the desired result. \\
We now have the continuity result for $r>1$. For a general power $r$, we proceed by induction. Fix $r>1$ and assume that for all $i=1,\cdots,r-1$,
$$\E^{1/i}[|Y_t-\bar Y_t|^i]\preceq C \max_{v=1,\cdots,p} \|\phi^{(v)}-\bar\phi^{(v)}\|^{1/i}_1\quad \text{for all } t\geq 1.$$
Note that the induction hypothesis implies that $\E^{1/i}[|\lambda_t-\bar \lambda_t|^i]\preceq C \max_{v=1,\cdots,p} \|\phi^{(v)}-\bar\phi^{(v)}\|_1^{1/i} $ for all $ t\geq 1.$ 
Conditioning on $\mathcal F_{t-1}$, we have that
\begin{align*}
    \E[|Y_t-\bar Y_t|^r]&=\E \left[ \E\left[ |N_t(\lambda_t)-N_t(\bar \lambda _t)|^r | \mathcal F_{t-1}\right]\right]\\
    &=\E \left[\E\left[ N_t(|\lambda_t-\bar \lambda _t|)^r | \mathcal F_{t-1}\right]\right]\\
    &=\E \left[|\lambda_t-\bar \lambda_t|^r + \sum_{i=1}^{r-1}  {i\brace r}|\lambda_t-\bar \lambda_t|^i\right],
\end{align*}
where $ {i\brace r}$ are the Stirling coefficients of second kind, \textit {cf.} \cite{Johnson}. Hence, using the induction hypothesis 
\begin{align*}
    \E[|Y_t-\bar Y_t|^r]&\preceq \E [|\lambda_t-\bar \lambda_t|^r] + \sum_{i=1}^{r-1}  {i\brace r} \max_{v=1,\cdots,p} \|\phi^{(v)}-\bar\phi^{(v)}\|_1^i C\\
    &\preceq \E [|\lambda_t-\bar \lambda_t|^r] +\max_{v=1,\cdots,p} \|\phi^{(v)}-\bar\phi^{(v)}\|_1 C.\\
\end{align*}
 Taking the power $1/r$ and using Minkowski's inequality, the fact that $(x+y)^{1/r}\leq  x^{1/r}+y^{1/r}$ and the Lipschitz continuity of $\psi$ we have
\begin{align*}
     \E^{1/r}[|Y_t-\bar Y_t|^r]\preceq &\E [|\lambda_t-\bar \lambda_t|^r]^{1/r} + \left( \max_{v=1,\cdots,p} \|\phi^{(v)}-\bar\phi^{(v)}\|_1\right)^{1/r} C\\
     \preceq & L \E \left[\left |\sum_{k=1}^{t-1}\phi^{(t)}_{t-k}Y_k -\bar \phi^{(t)}_{t-k}\bar Y_k \right|^r\right]^{1/r}+ \left( \max_{v=1,\cdots,p} \|\phi^{(v)}-\bar\phi^{(v)}\|_1\right)^{1/r} C \\
          \preceq & L \sum_{k=1}^{t-1} |\phi^{(t)}_{t-k}|\E^{1/r} [|Y_k-\bar Y_k|^r]\\
          &+ L \sum_{k=1}^{t-1} |\phi^{(t)}_{t-k}-\bar \phi^{(t)}_{t-k}|\E^{1/r}[|\bar Y_k|^r] + \left( \max_{v=1,\cdots,p} \|\phi^{(v)}-\bar\phi^{(v)}\|_1\right)^{1/r} C \\
     \preceq & L \sum_{k=1}^{t-1}A_{t-k}\E ^{1/r}[|Y_k-\bar Y_k|^{r}]\\
     &+ \left( \max_{v=1,\cdots,p} \|\phi^{(v)}-\bar\phi^{(v)}\|_1+\left( \max_{v=1,\cdots,p} \|\phi^{(v)}-\bar\phi^{(v)}\|_1\right)^{1/r}\right)C, \\
\end{align*}
where we used Proposition \ref{prop:stationary} to bound $\E^{1/r}[|\bar Y_k|^r]$ from above. We conclude using Lemma \ref{lmm:convolution}.\\
For Type II periodicity, we have using the same arguments that
\begin{align*}
     \E \left[|Y_t-\bar Y_t| \right]\preceq & L \sum_{k=1}^{t-1}A_{t-k}\E |Y_k-\bar Y_k| + L \sum_{k=1}^{t-1} |\phi^{(k)}_{t-k}-\bar \phi^{(k)}_{t-k}|C, \\
\end{align*}
which yield the desired result by bounding $\sum_{k=1}^{t-1} |\phi^{(k)}_{t-k}-\bar \phi^{(k)}_{t-k}|$ from above by $\sum_{v=1}^p\|\phi^{(v)}-\bar \phi^{(v)}\|_1$ and applying Lemma \ref{lmm:convolution}. 

\subsection{Proof of Theorem \ref{thm:4.3}}
\label{proof:4.3}
\begin{proof}
    For the first point, we use Lemma \ref{lmm:l1_desity} to find $q \in \N^*$ and $(G^{(m)}_v)_{v\in 1\in\{1,\cdots,p\}, m\in\{1,\cdots,q\}}$ such that 
    $$\max_{v\in \{1,\cdots,p\}}\left \|\phi^{(v)} - \sum_{m=1}^{q} G_v^{(m)} e^{-(2m+1)\frac{\cdot}{\tau}}\right\|_1\preceq \frac{\varepsilon}{d C} \mathds 1,$$
    where $C$ is the constant that appears in the first point of Proposition \ref{prop:continuity}. We point out that $\varepsilon$ can be chosen small enough to ensure that for any $v \in \{1,\cdots,p\}$ 
    $$\left|\sum_{m=1}G_v^{(m)} e^{-(2m+1)\frac{k}{\tau}} \right| \preceq A_k\quad \text{for all }k\geq 1.$$
    Let $\bar Y$ be the Poisson autoregression constructed using $N$ and the kernel $\bar \phi ^{(t)}_k=\sum_{m=1}^{(q)} G_t^{(m)} e^{-(2m+1)\frac{k}{\tau}}$. 
    Using Proposition \ref{prop:continuity} we have that 
    $$\E [|Y_t-\bar Y_t|]\preceq \varepsilon \mathds 1.$$
    If we define $\xi^{(m)}_t=\sum_{k=1}^{t-1}e^{-(2m+1)\frac{t-k}{\tau}}Y_k$, then we clearly have that $(\bar Y, \xi^{(1)},\cdots,\xi^{(q)})$ is a Markov chain.
For the second point, we construct $\bar Y$ according to 
\begin{equation*}
    \begin{cases}
    \bar Y_t &=N_t \left(\lambda_t\right)\\
    \lambda_t &= \psi \left (\mu_t+ \sum_{k=1}^{t-1} \sum_{m=1}^{q} J^{(m)}_{k} e^{-(2m+1)\frac{(t-k)}{\tau}}Y_k\right)
    \end{cases}   .
\end{equation*}
We set $$\zeta^{(m)}_t=\sum_{k=1}^{t-1}J^{(m)}_ke^{-(2m+1)\frac{t-k}{\tau}}Y_k,$$
and using the properties of the exponential we get
$$\zeta^{(m)}_t= e^{-\frac{2m+1}{\tau}} \zeta^{(m)}_{t-1} + e^{-\frac{2m+1}{\tau}} J_{t-1}^{(m)} Y_{t-1}.$$
The result follows from the continuity of Poisson autoregressions with respect to the kernel (Proposition \ref{prop:continuity}) and the density of the exponential polynomials in $\ell_1(\N^*)$ (Lemma \ref{lmm:l1_desity}).
\end{proof}

\subsection{Proof of Theorem \ref{thm:strong-cosistency}}
\label{sec:thm:strong-cosistency}

   The proof of the strong consistency of the MLE for a general integer valued autoregression has been established in \cite{AF} and extended to the case of time series with periodically changing coefficients in \cite{almohaimeed} and to the multivariate case with a distribution in the exponential family in \cite{SKK}. These articles also prove strong consistency even if the distribution used in the likelihood is mis-specified (\textit{e.g.} the use of a negative binomial MLE on data coming from Poisson autoregession). We point out that \cite{AF} and \cite{almohaimeed} prove the strong consistency for general autoregressions of the form 
    \begin{equation*}
        \begin{cases}
            \E[Y_t |\mathcal F^Y_{t-1}]&=\lambda_t(\theta^*),\\
            \lambda_t(\theta ^*)&=f_t(Y_{t-1},Y_{t-2},\cdots;\theta^*),
        \end{cases}
    \end{equation*}
    that encompasses the form we are dealing with in this article. It is then sufficient to show that \eqref{eq:def_likelihood} satisfies Assumptions A1-A6 in \cite{almohaimeed}. \\
    Without loss of generality, we focus on the univariate case. The multivariate case is inferred component by component. 
    Given that Assumption \ref{ass:mle} is in force and using the fact that $\psi$ is Lipschitz continuous, A1, A2, A5 and A6 are satisfied. Proposition \ref{prop:stationary} yields that once the stability assumption is met, $\tilde Y$ has moments of any order $r\geq 1$, hence A4 holds. We now show that A3 also holds, \textit{i.e.} that the initial values are asymptotically not important. Given a strictly periodically stationary observation $(\tilde Y_t)_{t \in \Z}$, let 
    $$\tilde \lambda_t(\gamma)=\psi\left(\mu_t +\sum_{k=-\infty}^{t-1}\phi^{(t)}_{t-k}\tilde Y_k\right) \quad \text{and}\quad  \lambda_t(\gamma)=\psi\left(\mu_t +\sum_{k=0}^{t-1}\phi^{(t)}_{t-k}\tilde Y_k\right),$$
    where $\phi$ are exponential polynomials. Using the Lipschitz continuity of $\psi $ we have that 
    \begin{align*}
        \left|\tilde \lambda_t(\gamma)- \lambda_t(\gamma) \right| &\leq L \sum_{k=-\infty}^{0} |\phi^{(t)}_{t-k}|\tilde Y_k\\
        &\leq L \sum_{k=-\infty}^0 \sum_{m=1}^{(q)} |G_t^{(m)}| e^{-m\frac{t-k}{\tau}} \tilde Y_k,
    \end{align*}
    and since $\Gamma$ is a compact set, 
    \begin{align*}
        \sup_{\gamma \in \Gamma} \left|\tilde \lambda_t(\gamma)- \lambda_t(\gamma) \right| &\leq C \sum_{k=-\infty}^0 e^{-\frac{t-k}{\tau}}\tilde Y_k,
    \end{align*}
    for some positive constant $C$. For any $\varepsilon >0$ we have using Markov's inequality 
    \begin{align*}
        \mathbb P \left(e^{\frac{t}{2\tau}}\sup_{\gamma \in \Gamma} \left|\tilde \lambda_t(\gamma)- \lambda_t(\gamma) \right|\geq \varepsilon\right) &\leq \frac{Ce^{\frac{t}{2\tau}} \sum_{k=-\infty}^0 e^{-\frac{t-k}{\tau}}\E[\tilde Y_k]}{\varepsilon}\\
        &\leq  \frac{C e^{\frac{t}{2\tau}}\sum_{k=t}^{+\infty}e^{-\frac{k}{\tau}}}{\varepsilon}\\
        &\leq \frac{Ce^{-\frac{t}{2\tau}}}{\varepsilon},
    \end{align*}
    which yields by Borel-Cantelli's Lemma the almost sure existence of a constant $C$ such that 
    $$\sup_{\gamma \in \Gamma} \left|\tilde \lambda_t(\gamma)- \lambda_t(\gamma) \right| \leq C e^{-\frac{t}{2\tau}}.$$
    We then have that 
    $$\tilde Y_t \sup_{\gamma \in \Gamma} \left|\tilde \lambda_t(\gamma)- \lambda_t(\gamma) \right| \leq C e^{-\frac{t}{2\tau}} \tilde Y_t,$$
    which again by applying Markov's Inequality and the Borel-Cantelli Lemma yields $$\lim_{t \to +\infty}\tilde Y_t \sup_{\gamma \in \Gamma} \left|\tilde \lambda_t(\gamma)- \lambda_t(\gamma) \right|=0, \quad \text{almost surely}$$ and therefore the desired result for the parameters $\mu_v$ and $G_{v}^{(m)}$. For the reconstructed kernels, we have that 
    \begin{align*}
        \|\phi^{(v)}-\phi^{(v)}_T\|_1 &= \sum_{k=1}^{+\infty} |\phi^{(v)}_k-\phi^{(v)}_k|\\
        &\preceq  \sum_{m=1}^{(q)} |G^{(m)}_v-G^{(m)}_{v,T}|\sum_{k=1}^{+\infty} e^{-(2m+1)\frac{k}{\tau}},
    \end{align*}
    which tends to zero almost surely as $T$ goes to infinity.

\section{Preliminary lemmas and proofs}
\begin{Lemma}
    \label{lmm:spectral}
    Let $A$ and $B$ be two matrices with non-negative coefficients. If $A \preceq B$, then $\rho(A)\leq \rho(B)$.
\end{Lemma}
\begin{proof}
Since the matrix product is a combination of sums and products of the coefficients, we have that $A^j \preceq B^j$ for any $j\geq 1$. By taking $\|\cdot\|_1$ to be the maximum column sum norm and using the fact that the function $x\mapsto x^{1/j}$ is increasing we have that 
    $$\|A^j\|_1^{1/j} \leq \|B^j\|_1^{1/j}.$$
    We now let $j\to +\infty$ and use Gelfand's formula to obtain that $$\rho \left (A\right) \leq \rho \left (B\right) .$$ 
\end{proof}
\begin{Lemma}
    \label{lmm:convolution}
    Let $(A_k)_{k\in \N}$ be a family of non-negative matrices satisfying $\rho \left (\sum_{k=1}^{+\infty}A_j\right)<1$. If a  sequence $(x_t)_{t\in \Z}$ that takes finite values satisfies 
    $$|x_t|_{ }\preceq \sum_{k=1}^{+\infty}A_k |x_{t-k}|_{ } + |K_t|_{ }$$
    for any $t\in \Z$ and for some vector $(K_t)_{t \in \Z}$, then

    $$|x_t| \preceq \left (\left(\sum_{m=0}^{+\infty}  A ^{*m} \right)*|K|\right)_t.$$
    In particular, if there exists a nonnegative constant vector $\bar{K}$ such as $K_t \preceq \bar K$ then
    $$|x_t|_{ }\preceq \left (I-\sum_{k\geq 1}A_k \right)^{-1}|\bar K|_{ }.$$
\end{Lemma}
\begin{proof}
    First, set $\tilde A_k = A_k \boldsymbol{1}_{k\geq 1}$ for any $k \in \Z$. We recall that the discrete convolution operator 
    $$(a*b)_t=\sum_{j=-\infty}^{+\infty} a_jb_{t-j}$$
    is associative and that for $a$ and $b$ summable families of matrices $\sum_{t=-\infty}^{+\infty} (a*b)_t = \left(\sum_{t=-\infty}^{+\infty} a_t \right)\left(\sum_{t=-\infty}^{+\infty} b_t \right)$. \\
    In particular, this means that for a given $m \in \N$, $\sum_{k=-\infty}^{+\infty} \tilde A^{*m}_k = \left(\sum_{k=-\infty}^{+\infty}\tilde A _k\right)^{(m)}$, where $\tilde A ^{*m}$ is given recursively by $\tilde A^{*0}=\delta_k I$ and $\tilde A^{*(k+1)}=\tilde A * \tilde A^{*k}$.
    When can then write for any $t \in \Z$ 
    $$|x_t|_{ }\preceq \left(\tilde A*|x|_{ }\right)_t + |K|_{ }.$$
    For any given $m\in \N$, taking the convolution by $\tilde A ^{*m}$ to the right and the left and rearranging the terms yield
    \begin{align}
    \label{ineq:convolution}
        \left(\tilde A^{*m}*|x|_{ }\right)_t-\left(\tilde A^{*(m+1)}*|x|_{ }\right)_t \preceq \left(\tilde A ^{*m}*|K|_{ }\right)_t.
    \end{align}
    Summing the last inequality for $m$ ranging from $0$ to some $n \in \N$ yields 
    $$|x_t|_{ }-\left(\tilde A ^{*(n+1)}*|x|_{ }\right)_t \preceq \sum_{m=0}^{n}\left(\tilde A ^{*m}*|K|_{ }\right)_t.$$
    Since $\rho \left(\sum_{k=-\infty}^{+\infty}\tilde A_k\right)<1$, the matrix $B:=\sum_{n\geq 0} \tilde A^{*n}$ is well-defined and $$\sum_{k=-\infty}^{+\infty} B_k= (I-\sum_{k=-\infty}^{+\infty}\tilde A_k).$$ Hence, letting $n$ go to infinity in \eqref{ineq:convolution} and noticing that $\lim_{n\to +\infty}\tilde A^{*n}=0$ yields the result.  
\end{proof}

 \begin{Lemma}
   \label{lmm:l2_density}
       Let $(\phi_k)_{k\in \N}$ be a square integrable sequence and fix $\tau >0$. For a given $\varepsilon >0$, there exists $r \in \N^*$ and $(\nu^r_1, \cdots, \nu^r_r) \in \R ^ r$ such that 
       $$\left ( \sum_{k=0}^{+\infty} \left( \phi_k-\sum_{m=1}^r \nu_m^r e^{-2\frac{k m}{\tau}} \right)^2\right)^{1/2} \leq \varepsilon. $$
   \end{Lemma}
   \begin{proof}
       We show that the linear span of $\left \{\left ( e^{-2\frac{nk}{\tau}}\right)_{k\in \N}; n\in \N \right \}$ is dense in $\ell_2$. To do so, we show that the only square integrable sequence $(\phi_k)_{k\in N}$ to satisfy
       \begin{equation}
       \label{density}
       \left \langle \phi, e^{-2\frac{n \cdot}{\tau}} \right \rangle=0, \quad \text{for all } n \in \N^* 
       \end{equation}
       is the sequence $\phi=0$. Let $\phi \in \ell _2$ satisfying \eqref{density}. Define the function $$f(x)=\sum_{k=0}^{+\infty}\phi_{k}x^k$$ on $(-1,1)$. $f$ is continuous in the vicinity of $0$ ($\phi$ being square integrable) and satisfies 
       $$f \left(e^{-2\frac{n}{\tau}}\right)=0, \quad \text{for all $n \in \N$},$$
       thus by letting $n \to +\infty$, $f(0)=0$. From that we deduce that $$\phi_0=0.$$
       Let $j$ be a positive integer and suppose that $\psi_0=\cdots=\psi_{j-1}=0$ and define $f_j(x)=\sum_{k=0}^{+\infty}\phi_{k+j}x^k$ (continuous at zero). For $n \in \N$ we have that 
       \begin{align*}
           f_j(e^{-2\frac{n}{\tau}})&=\sum_{k=0}^{+\infty} \phi_{k+j}e^{-2\frac{n k}{\tau}}\\
           &=e^{2\frac{nj}{\tau}} \left \langle \phi, e^{-2\frac{n\cdot}{\tau}}\right \rangle\\
           &=0.
       \end{align*}
       Using continuity, we have that $f_j(0)=0$, yielding $\phi_j=0$. The result then follows by induction. 
   \end{proof}

   \begin{Lemma}
       \label{lmm:l1_desity}
       Let $(\phi_k)_{k\in \N}$ be an $\ell_1$ sequence and fix  $\tau >0$. For a given $\varepsilon >0$, there exists $r \in \N^*$ and $(\nu^r_1, \cdots, \nu^r_r) \in \R ^ r$ such that 
       $$ \sum_{k=0}^{+\infty} \left |\phi_k-\sum_{m=1}^r \nu_m^r e^{-(2m+1)\frac{k }{\tau}} \right| \leq \varepsilon. $$
   \end{Lemma}
   \begin{proof}
       Let $\varepsilon>0$. Using the fact that the remainder of the sum of an $\ell_1$ series tends to zero, we can find $T>0$ such that 
       \begin{align*}
           \sum_{k=0}^{+\infty} |\phi_k-\phi_k \boldsymbol{1} _{k\leq T}| \leq \frac{\varepsilon}{2}. 
       \end{align*}
       The sequence $(\phi_k \boldsymbol{1}_{k\leq T}e^{\frac{k}{\tau}})_{k\in \N}$ is of a finite support and hence square integrable. Using Lemma \ref{lmm:l2_density} we have that
       \begin{equation}
           \label{ineq:l2}
           \left(\sum_{k=0}^{+\infty} \left( \phi_k \boldsymbol{1}_{k\leq T} e^{\frac{k}{\tau}}-\sum_{m=1}^r \nu_m^r e^{-2\frac{k m}{\tau}}\right)^2\right)^{1/2} \leq \frac{\varepsilon \sqrt{1-e^{-2\tau^{-1}}}}{2},
       \end{equation}
       for some $r\in \N^*$ and $(\nu_1^r,\cdots,\nu_r^r)$. Thus, using Cauchy-Schwarz's inequality we have that 
       \begin{align}
           \sum_{k=0}^{+\infty} \left|\phi_k \boldsymbol{1}_{k\leq T}-\sum_{m=1}^r \nu_m^r e^{-(2m+1)\frac{k }{\tau}} \right| &=\sum_{k=0}^{+\infty} \left|\phi_k \boldsymbol{1}_{k\leq T}e^{\frac{k}{\tau}}-\sum_{m=1}^r \nu_m^r e^{-2m\frac{k }{\tau}} \right| e^{-\frac{k}{\tau}} \nonumber\\
           & \leq \left( \sum_{k=0}^{+\infty} \left|\phi_k \boldsymbol{1}_{k\leq T}e^{\frac{k}{\tau}}-\sum_{m=1}^r \nu_m^r e^{-2m\frac{k }{\tau}} \right|^2 \right) ^{1/2} \left (\sum_{k=0}^{+\infty} e^{-2\frac{k}{\tau}} \right)^{1/2} \label{ineq:bound_l_2}\\
           &=\left( \sum_{k=0}^{+\infty} \left|\phi_k \boldsymbol{1}_{k\leq T}e^{\frac{k}{\tau}}-\sum_{m=1}^r \nu_m^r e^{-2m\frac{k }{\tau}} \right|^2 \right) ^{1/2} \frac{1}{\sqrt{1-e^{-2\tau^{-1}}}}\nonumber \\
           &\leq \frac{\varepsilon}{2} \quad \text{using \eqref{ineq:l2}} \nonumber .
       \end{align}
       And finally, 
       \begin{align*}
           \sum_{k=0}^{+\infty} \left|\phi_k -\sum_{m=1}^r \nu_m^r e^{-(2m+1)\frac{k }{\tau}} \right| &\leq \sum_{k=0}^{+\infty} \left|\phi_k - \phi_k \boldsymbol{1}_{k\leq T} \right| + \sum_{k=0}^{+\infty} \left|\phi_k  \boldsymbol{1}_{k\leq T}-\sum_{m=1}^r \nu_m^r e^{-(2m+1)\frac{k }{\tau}} \right|\\
           & \leq \frac{\varepsilon}{2} + \frac{\varepsilon}{2}\\
           &\leq \varepsilon.
       \end{align*}
   \end{proof}

\end{document}